\documentclass[11pt,times,letter]{article}
\usepackage{amsmath, amsfonts,amsthm}
\usepackage{dsfont}
\usepackage{bbm}

\usepackage[english]{babel}

\usepackage[normalem]{ulem}

\usepackage{latexsym,amssymb,graphicx}
\usepackage{verbatim}
\usepackage{mathrsfs}
\usepackage{epsfig}

\usepackage{xcolor}
\usepackage{tikz}
\usepackage{pgfplots}
\usepgfplotslibrary{fillbetween}

\usepackage{url}
\usepackage{bm}
\usepackage{natbib}

\usepackage[colorlinks,linkcolor=blue,citecolor=blue,anchorcolor=blue]{hyperref}

\def\esssup_#1{\underset{#1}{\mathrm{ess\,sup\, }}}
\def\essinf_#1{\underset{#1}{\mathrm{ess\,inf\, }}}
\def\argmax_#1{\underset{#1}{\mathrm{arg\,max\, }}}
\def\argmin_#1{\underset{#1}{\mathrm{arg\,min\, }}}

\newtheorem{theorem}{Theorem}[section]
\newtheorem{definition}{Definition}[section]
\numberwithin{equation}{section}

\newtheorem{proposition}[theorem]{Proposition}
\newtheorem{assumption}[theorem]{Assumption}
\newtheorem{remark}[theorem]{Remark}
\newtheorem{lemma}[theorem]{Lemma}

\allowdisplaybreaks[4]

\definecolor{Red}{rgb}{1.00, 0.00, 0.00}

\definecolor{DRed}{rgb}{0.5, 0.00, 0.00}

\definecolor{Blue}{rgb}{0.00, 0.00, 1.00}

\definecolor{Green}{rgb}{0.0, 0.4, 0.0}

\definecolor{mycolor1}{rgb}{0.00000,0.44700,0.74100}%
\definecolor{mycolor2}{rgb}{0.85000,0.32500,0.09800}%
\definecolor{mycolor3}{rgb}{0.92900,0.69400,0.12500}%
\definecolor{mycolor4}{rgb}{0.49400,0.18400,0.55600}%
\definecolor{mycolor5}{rgb}{0.46600,0.67400,0.18800}%
\definecolor{mycolor6}{rgb}{0.00000,1.00000,1.00000}%

\pgfplotsset{compat=newest}

\title{Graphon Mean Field Game of mutual holding}

\author{Daorong Cui \thanks{Email: dcuiab@connect.ust.hk, Department of Mathematics, The Hong Kong University of Science and Technology, Clearwater Bay, Kowloon, Hong Kong.}
\and
Shuoqing Deng \thanks{Email: masdeng@ust.hk, Department of Mathematics, The Hong Kong University of Science and Technology, Clearwater Bay, Kowloon, Hong Kong. S. Deng is supported by the Hong Kong University of Science and Technology Start-up grant no. R9826 and Hong Kong RGC Early Career Scheme (ECS) under grant no. 26307125.}
\and
Yang Xiang\thanks{Email: maxiang@ust.hk,  Department of Mathematics, The Hong Kong University of Science and Technology, Clearwater Bay, Kowloon, Hong Kong.}
}
\date{\vspace{-1cm}}

\begin{document}
\maketitle

\begin{abstract}
 
This paper studies the mean field game of mutual holding proposed by Djete and Touzi(AAP, 2024), and consider the case where the interactions among agents are described by a graphon.  We adopt the formulation on the enlarged space which is modeled using the joint law of the value process and the graphon label, as in Lacker and Soret(MOR, 2023).  Under suitable conditions on the graphon function, we are able to provide the explicit characterization of the optimal strategy, prove the wellposedness of associated Mckean-Vlasov SDE and establish the convergence results of the Nash equilibria. The key technique consists in a detailed analysis of the continuity property under the $\mathcal{WOP}_2$ metric, and tailor-made arguments for different graphon equilibria under different regularities of the model.

\vspace{0.6 cm}


\vspace{0.2 cm}

\end{abstract}

\section{Introduction} \label{sec:intro}
We study the strategic interactions between economic agents due to the equity cross-holding behaviors for the purpose of risk diversification, proposed by Djete and Touzi \cite{DjeteTouzi2024}. Each agent has a prespecified indiosyncratic risk process, and seeks to maximize her utility as a function of the final wealth value. When the number of agents optimizing their objectives tends to infinite, this leads to a limiting mean field game of optimal mutual holding, for which they provided a semi-explicit equilibrium. In particular, the optimal control is of bang-bang type when there is no common noise. 

This problem was later further explored by \cite{DjeteGuo2023, {DjeteBassou2025}}. Djete et al. \cite{DjeteGuo2023} considered the default probability by introducing absorption at the origin of the equity process. They provided a semi-explicit equilibrium, and proved the existence of a unique strong solution for the equilibrium McKean-Vlasov SDE using the parametrix arguments under suitable assumptions. In \cite{DjeteBassou2025}, the authors considered the case with common noise. As in the current context the holding position affects both the drift and the diffusion terms, the optimal control is no longer of bang-bang type, and the problem exhibits the standard paradigm of mean-variance trade-off. They then further reduced the problem to a portfolio optimization problem with random endowment, and provided some characterizations of the optimal solution in the context of the Black-Scholes model with respectively the power and logarithmic utility function.

 In the present paper, our main modeling assumption is that: economic agents interact with each other heterogeneously, through positive finite measures defined as the weighted average of agents' states. In the standard setting, the analysis of such systems falls in the framework of mean field game (MFG), introduced independently in the seminal works of Lasry and Lions \cite{LasryLions2007} and Huang et al. \cite{Huang2006}. However, the standard symmetry hypothesis hinders the application of MFG in many systems of practical and theoretical interest, in which players may belong to different communities or classes, interact asymmetrically or encounter other sources of heterogeneity. 

Motivated by the limitations of MFG framework and thanks to the development of graphon theory by Lov\'asz \cite{Lovasz2012}, graphon-based continuum models have gained considerable attention in recent years. Specifically, a graphon $G$ is a symmetric measurable function: $[0,1]^2 \rightarrow [0,1]$.  We refer to \cite{CoppCres2025, BayrWu2022,BayrHe2025,Bayr2023,Wu2023} for the recent work on graphon particle systems,  \cite{LackerSoret2023, Tangpi2024, Zhang2023,Caines2021,Aurell2022} for graphon mean field games, and \cite{Djete2025, Crescenzo2026} for graphon mean field control. To justify the effectiveness of MFG limit, Lacker and Soret \cite{LackerSoret2023} used their graphon game equilibrium to construct an approximate Nash equilibrium for the N-player game with the underlying graph sequence converging in the strong operator topology. Similarly, an $\epsilon$-Nash equilibrium result was established by Huang et al. in \cite{Caines2021}. On the other hand, Bayraktar et al. \cite{Zhang2023} and Tangpi and Zhou \cite{Tangpi2024} studied the general convergence problem as an application of (backward) propagation of chaos results for systems of interacting forward-backward SDEs. 

As in \cite{LackerSoret2023}, our paper adopts the unified modeling for agents' labels by enlarging the state space from $\mathbb{R}$ to $[0, 1] \times \mathbb{R}$. It is known that working directly with a continuum of agents driven by a continuum of Brownian motions $(B^u)_{u \in [0, 1]}$, may raise significant technical issues since $(u, \omega) \mapsto B^u(\omega)$ is not jointly measurable in the usual product space, see \cite{Sun1998}. In \cite{Bayr2023,BayrWu2022, Wu2023}, this issue was avoided by arguing that the map $u \mapsto \mathcal{L}(X^u)$ is measurable, or the map $u \mapsto \mathbb{E}[X^u_t]$ is measurable \cite{Crescenzo2026}. Bayraktar et al. \cite{Zhang2023} and Amini et al. \cite{Amini2025} addressed this measurability issue by working on one stochastic basis, i.e. taking one driven Brownian motion for all labels. Aurell, Carmona and Lauriere \cite{Aurell2022}, followed by \cite{BayrHe2025, Tangpi2024, CoppCres2025} etc., directly tackled this issue by setting the model on a ``Fubini extension'' of the usual product between the sample probability space and a probability space extending the Lebesgue space $([0,1], \mathcal{B}([0,1]), \lambda_{[0,1]})$, see \cite{Sun2006} for a self-contained presentation. By working on the enlarged state space as \cite{LackerSoret2023}, we can bypass this measurability issue. Indeed, this is also natural for out purpose since the limiting dynamics in the mutual holding problem do not have a unique strong solution due to the potential singularities in the coefficients.

The main contributions of our work are two folds. First, we adopt the 2-Wasserstein On Positive ($\mathcal{WOP}_2$) metric to tackle the positive finite measures representing the heterogeneous interactions among agents, and carry out a detailed analysis of the continuity properties related to the $\mathcal{WOP}_2$ metric, see Section \ref{subsec:tech} and Section \ref{E:sec:1}. Such a metric was also employed by Bayraktar et al. in \cite{BayrHe2025}. 
 Introduced by \cite{LeblGoui2023}, $\mathcal{WOP}_2$ metric enables us to remove the non-degenerate conditions imposed on $G$ as in \cite{CoppCres2025}, and facilitate the study of the continuity property of interested quantities later, i.e. $B$ and $c$, defined in \eqref{E:eqn:3} and \eqref{E:eqn:52}. It can be seen as a more general setting as the scalar interaction framework in \cite{Bayr2023,Wu2023,Zhang2023,BayrWu2022}.

Second, under the label-state framework, we derive an explicit optimal solution, and verify the existence of a weak solution for the equilibrium SDE under the continuity assumption for $G$ on $[0,1]^2$. In particular, we then establish the approximate Nash equilibrium result for the finite-agent game with the approximation error vanishing in an averaged sense, both for the sampling kernel and the general kernel with the assumption that the interaction matrix stems from $G$ or $N^2 \|G^N -G\| \rightarrow 0$.

The rest of the paper is organized as follows. We conclude Section \ref{sec:intro} with useful notations used throughout the paper. In Section \ref{sec:results}, we provide the formulation of the problem, and present the main results of the paper: namely the explicit solution, the well-posedness of the SDE and the propogation of chaos result. In Section \ref{sec:wellpose} and Section \ref{E:sec:1}, we give the detailed proofs of the main results.

\subsection{Notations}
\noindent \textbf{Measures.} \ \ For a measurable space $E$, we denote $\mathcal{M} (E)$, $\mathcal{M}^2 (E)$, $\mathcal{P} (E)$ and $\mathcal{P}^2 (E)$ respectively the collection of finite positive measures,  finite positive measures with finite second moment, probability measures and probability measures with finite second moment on $E$. We further define $\mathcal{P}_{\text{Unif}}([0,1]\times E)$ as the collection of probability measures on $[0,1]\times E$ with uniform first marginal, and $\mathcal{P}^p_{\text{Unif}}([0,1]\times E)$ the subset of $\mathcal{P}_{\text{Unif}}([0,1]\times E)$ with finite $p$-th moment for some $p>2$.

 In addition, we introduce $m_\mu$ as the total mass of measure $\mu$, and $\bar{\mu} = \mu/m_\mu$ if $m_\mu > 0$ and $\delta_{x_0}$ otherwise, where $\delta_{x_0}$ denotes the Dirac measure at $x_0$. For $x_0 \in E, a>0$, we further define $M_{x_0} (\mu) := \int_E |x - x_0|^2 \mu(\mathrm{d} x)$ and $T_a (x): = a (x - x_0) + x_0$. The pushforward of a measure $\mu$ by a measurable mapping $T$ is denoted by $T\#\mu$.
 
We remind for a signed measure $M$, we have $|M| = M^+ + M^-$, where $M^+$ and $M^-$ are respectively the positive and the negative parts of $M$. In particular, there exist two measurable sets $P$ and $N$ such that:
\begin{itemize}
    \item $P \cup N = E$ and $P \cap N = \emptyset$,
    \item $P$ is a positive set and $N$ is a negative set,
\end{itemize}
and for every measurable set $B$, we have $M^+(B) = M(B \cap P)$ and $M^-(B) = -M(B \cap N)$.

\textbf{Metrics and norms.} \ \ We equip $\mathcal{P}^2 (E)$ with the 2-Wasserstein distance $\mathcal{W}_2$ so that it is a metric space itself. The 2-Wasserstein distance on $\mathcal{P}^2 (E)$ is defined for $\mu, \nu \in \mathcal{P}^2(E)$
\begin{equation*}
    \mathcal{W}_2 (\mu, \nu) = \inf_{\gamma \in \Pi (\mu, \nu)} \left(\int_{E \times E} |x - y|^2 \gamma(\mathrm{d}x, \mathrm{d}y)\right)^{1/2},
\end{equation*}
where $\Pi (\mu, \nu)$ denotes the set of probability measures in $\mathcal{P} (E \times E)$ with the first marginal $\mu$ and the second marginal $\nu$. One can also find a control on the Wasserstein distance, see (equation (1.3), \cite{CoppCres2025}), i.e. for any $\mu, \nu \in \mathcal{P}^2(E)$:
\begin{equation}
    \mathcal{W}^2_2(\mu,\nu) \leq 2 \int_E |x-x_0|^2|\mu - \nu| (\mathrm{d}x), \quad \text{for some $x_0 \in E$},
    \label{E:eqn:28}
\end{equation}
where $|\mu-\nu|$ is the variation of the (signed) measure $\mu - \nu$. 

 Let $E = \mathbb{R}^d$. Fix an arbitrary reference point $x_0 \in \mathbb{R}^d$. For any $\mu, \nu \in \mathcal{M}^2 (\mathbb{R}^d)$, the 2-Wasserstein On Positive measures ($\mathcal{WOP}_2$) metric is defined by
\begin{align*}
    \mathcal{WOP}_2^2 (\mu, \nu) & = (m_\mu - m_\nu)^2 + \mathcal{W}_2^2 (T_{m_\mu}\#\bar{\mu}, T_{m_{\nu}}\#\bar{\nu}) \\
    & = (m_\mu - m_\nu)^2 + (m_\mu - m_\nu) \left(M_{x_0} (\mu) - M_{x_0}(\nu)\right) + m_\mu m_\nu \mathcal{W}^2_2 (\bar{\mu}, \bar{\nu}),
\end{align*}
where the second equality comes from equation (7) in \cite{LeblGoui2023}.
    
 Denote by $C([0,T], \mathbb{R}^d)$ the set of continuous functions from $[0, T]$ to $\mathbb{R}^d$, endowed with the topology of uniform convergence. Let $\mathcal{C}^d := C([0, T], \mathbb{R}^d)$ and $\|x\|_t := \sup_{0 \leq s \leq t} |x_s|$ for $x$ in $\mathcal{C}^d$ and $t \in [0, T]$. We simply write $\mathcal{C}$ when $d=1$.

\textbf{Graphon functions.} \ \  Let $I := [0, 1]$. Denote by $\mathcal{G}$ the space of all bounded symmetric measurable functions $G: I \times I \rightarrow \mathbb{R}$. A graphon $G$ is an element of $\mathcal{G}$ with $0 \leq G \leq 1$. The cut-norm on $\mathcal{G}$ is defined by
\begin{equation*}
    \|G\|_\square := \sup_{S, T \in \mathcal{B}(I)} \left|\int_{S \times T} G(u,v)\mathrm{d}u\mathrm{d}v\right|.
\end{equation*}

 The step graphon $G^N$ associated with an $N \times N$ matrix $(\xi^N_{i j})_{1 \leq i,j \leq N}$ is defined as follows:
\begin{equation}
    G^N (u,v) := \xi^N_{i j}, \quad \text{for } (u,v) \in I^N_i \times I^N_j,
    \label{E:eqn:21}
\end{equation}
where $I^N_i := [(i-1)/N, i/N)$, for $i =1 , \ldots, N-1$, and $I^N_N := [(N-1)/N, 1]$.
    
For a given $\mu \in \mathcal{P}([0,1]\times E)$, we define the map $[0,1] \ni u \mapsto [G \mu](u) \in \mathcal{M}(E)$ as:
$$
[G \mu](u):=\int_I G(u,v) \mu(\mathrm{d}v, \mathrm{d}x).
$$

\section{Main results} \label{sec:results}
\subsection{Problem formulation}
\subsubsection{N-agent game}
Denote by $\mathbf{u}:=(u^1, \cdots, u^N) \in I^N_1 \times \cdots \times I_N^N$ the labels of $N$ agents. The states of these agents/particles are governed by the following dynamics:
\begin{equation*}
    \mathrm{d}X_t^i = \frac{1}{N}\sum_{j=1}^N\pi_t^{i,j}\mathrm{d}X_t^j - \frac{1}{N}\sum_{j=1}^N\pi_t^{j,i}\mathrm{d}X^i_t + b(t, X_t^i, M^{i,N}_t)\mathrm{d}t + \sigma(t, X^i_t, M^{i,N}_t)\mathrm{d}W^i_t,
\end{equation*}
where the neighborhood empirical measure $M^{i,N}_t$ which reflects the influence of the other agents on agent $i$, is defined by
\begin{equation}
    M^{i,N}_t := [G^N \mu^N_t](u^i) = \frac{1}{N}\sum_{j=1}^N \xi^N_{i j} \delta_{X^j_t}, \text{ with } \mu^N_t := \frac{1}{N}\sum_{j=1}^N \delta_{(u^j, X_t^j)}.
    \label{E:eqn:16}
\end{equation}

\subsubsection{Equilibrium dynamics}
Let $T>0$ and $\widehat{\Omega}:=[0,1]\times \mathcal{C}$. Denote by $(\widehat{U}, \widehat{X})$ the canonical elements on $\widehat{\Omega}$ and by $\widehat{\mathbb{F}}^0:=\{\widehat{\mathcal{F}}^0_t, t\in [0,T]\}$ the corresponding natural filtration, with $\widehat{\mathcal{F}}^0_t:=\sigma\{\widehat{U}; \widehat{X}_s, s\leq t\}$. Define further $\widehat{\mathbb{F}}:=\{\widehat{\mathcal{F}}_t, t \in [0,T]\}$ to be the filtration given by $\widehat{\mathcal{F}}_t := \lim_{s \downarrow t}\widehat{\mathcal{F}}^U_s$ with $\widehat{\mathcal{F}}^U_t := \cap_{\mu \in \mathcal{P}(\widehat{\Omega})} (\widehat{\mathcal{F}}^0_t)^\mu$, where $(\widehat{\mathcal{F}}^0_t)^\mu$ denotes the completion of $\widehat{\mathcal{F}}^0_t$ by $\mu$.

We also fix some initial distribution $\nu \in \mathcal{P}^p_{\text{Unif}}([0,1]\times\mathbb{R})$, with $p>2$. Let $\mathcal{P}_\mathcal{S} \subset \mathcal{P}(\widehat{\Omega})$ be the collection of all probability measures $\mu$ on $\widehat{\Omega}$ such that $\mu \circ (\widehat{U})^{-1} = \text{Unif}([0,1])$, and $\widehat{X}$ is a $\mu$-square integrable It\^o process, i.e.
$$
\widehat{X}_t := \widehat{X}_0 + \int_0^t B^\mu_s \mathrm{d}s + \int_0^t \Sigma^\mu_s\mathrm{d} W_s^\mu, \quad t \in [0, T], \quad \text{$\mu$-a.s. with $\mu \circ (\widehat{U}, \widehat{X}_0)^{-1} = \nu$}
$$
for some $\mu$-Brownian motion $W^\mu$, and some $\widehat{\mathbb{F}}$-progressively measurable deterministic functions $B^\mu, \Sigma^\mu: [0,T]\times \widehat{\Omega} \rightarrow \mathbb{R}$ satisfying
$$
\widehat{\mathbb{E}}^\mu \left[\int_0^T |B^\mu_s|^2 + |\Sigma^\mu_s|^2\mathrm{d}s\right] < \infty,
$$
where $\widehat{\mathbb{E}}^\mu$ denotes the corresponding expectation operator on $(\widehat{\Omega}, \widehat{\mathbb{F}}, \mu)$.

An admissible mutual holding strategy is defined by a measurable function
$$
\pi: [0,T] \times ([0, 1] \times \mathbb{R})^2 \longrightarrow \mathbb{R}.
$$
We denote by $\mathcal{A}$ the collection of all such maps. We fix some distribution $\mu \in \mathcal{P}_\mathcal{S}$ and consider for all such holding strategy $\pi \in \mathcal{A}$ the following mean-field SDE driven by a $\mathbb{P}$-Brownian motion $W$ on a filtered complete probability space $(\Omega, \mathcal{F}, \mathbb{F}:=\{\mathcal{F}_t\}_{t \in [0, T]}, \mathbb{P})$:
\begin{align*}
    X_t = X_0 &+ \widehat{\mathbb{E}}^\mu \left[\int_0^t \pi(s, U,X_s,\widehat{U},\widehat{X}_s)\mathrm{d}\widehat{X}_s\right] - \int_0^t \widehat{\mathbb{E}}^\mu \left[\pi(s, \widehat{U}, \widehat{X}_s, U, X_s)\right] \mathrm{d}X_s \\
    & + \int_0^t b(s, X_s, [G\mu_s](U))\mathrm{d}s + \int_0^t \sigma(s, X_s, [G \mu_s](U))\mathrm{d}W_s,
\end{align*}
or in differential form:
\begin{equation}
    \mathrm{d}X_t = \frac{b(t, X_t, [G\mu_t](U)) + \widehat{\mathbb{E}}^\mu \left[\pi(t, U, X_t, \widehat{U}, \widehat{X}_t) B^\mu(t, \widehat{U}, \widehat{X}_t)\right]}{1+\widehat{\mathbb{E}}^\mu\left[\pi(t, \widehat{U}, \widehat{X}_t, U, X_t)\right]}\mathrm{d}t + \frac{\sigma(t, X_t, [G\mu_t](U))}{1+\widehat{\mathbb{E}}^\mu\left[\pi(t, \widehat{U}, \widehat{X}_t, U, X_t)\right]}\mathrm{d}W_t,
    \label{E:eqn:53}
\end{equation} 
where $\mu_t := \mathcal{L}(U,X_t)$. Let $\mathcal{P}_\mathcal{S}(\pi)$ be the subset consisting of all measures $\mu \in \mathcal{P}_\mathcal{S}$ such that the last SDE has a weak solution $X$ satisfying $\mu = \mathcal{L}(U, X)$. By identification, the drift and diffusion coefficients must hold with:

\begin{align}
    B^\mu(t,u,x) &:= \frac{b(t,x,[G\mu](u)) + \int \pi(t,u,x,v,y) B^\mu(t,v,y)\mu(\mathrm{d}v, \mathrm{d}y)}{1+\int \pi(t,v,y,u,x)\mu(\mathrm{d}v, \mathrm{d}y)}, \label{E:eqn:1} \\
    \Sigma^\mu(t,u,x) &:= \frac{\sigma(t,x,[G\mu](u))}{1+\int \pi(t,v,y,u,x)\mu(\mathrm{d}v, \mathrm{d}y)} \label{E:eqn:2}.
\end{align}

\subsubsection{Graphon mean field game formulation}
For all $\pi \in \mathcal{A}$ and $\mu \in \mathcal{P}_\mathcal{S}(\pi)$, the deviation of the representative agent from the mutual holding strategy $\pi \in \mathcal{A}$ to an alternative one $\beta \in \mathcal{A}$ is defined by introducing an equivalent probability measure $\mathbb{P}^\beta_{\pi, \mu}$ via the density w.r.t. $\mathbb{P}$:
$$
\frac{\mathrm{d} \mathbb{P}^\beta_{\pi, \mu}}{\mathrm{d} \mathbb{P}} := \exp\left(\int_0^T \psi_t \mathrm{d}W_t - \frac{1}{2}|\psi_t|^2\mathrm{d}t\right), \quad \text{on $\mathcal{F}_T$},
$$
where the process $\psi = \psi^\beta_{\pi, \mu}$ is defined by
$$
\psi_t := \frac{\widehat{\mathbb{E}}^\mu \left[(\beta - \pi)(t,U,X_t,\widehat{U},\widehat{X}_t)B^\mu(t,\widehat{U},\widehat{X}_t)\right]}{\sigma(t,X_t,[G\mu_t](U))} = \frac{\int (\beta-\pi)(t,U,X_t,v,y)B^\mu(t,v,y)\mu_t(\mathrm{d}v, \mathrm{d}y)}{\sigma(t,X_t,[G\mu_t](U))}.
$$

We shall assume below that $\sigma$ is bounded from below away from zero, so that the above change of measure is well-defined by the boundedness of $\pi$ and $\beta$ and the $\mu$-square integrability of $B^\mu$. Moreover, it follows from the Girsanov theorem that the process $W^{\pi,\mu,\beta}_\cdot := W_\cdot - \int_0^\cdot \psi_t\mathrm{d}t$ is a $\mathbb{P}^\beta_{\pi, \mu}$-Brownian motion, so that the $\mathbb{P}^\beta_{\pi, \mu}$-dynamics of the value process $X$ are given by
\begin{align*}
    X_t = &X_0 + \int_0^t \Big(b(s,X_s, [G\mu_s](U)) + \widehat{\mathbb{E}}^\mu\left[\beta(s,U,X_s,\widehat{U}, \widehat{X}_s) B^\mu (s, \widehat{U}, \widehat{X}_s)\right]\Big) \mathrm{d}s \\
    & - \int_0^t \widehat{\mathbb{E}}^\mu \left[\pi(s,\widehat{U}, \widehat{X}_s, U, X_s)\right] \mathrm{d}X_s + \int_0^t \sigma(s, X_s, [G\mu_s](U))\mathrm{d}W^{\pi, \mu, \beta}_s.
\end{align*}
The representative agent seeks for an optimal mutual holding strategy by maximizing her criterion
$$
J_{\pi, \mu}(\beta) := \mathbb{E}^{\mathbb{P}^\beta_{\pi, \mu}}[g(X_T)] \quad \text{over all $\beta \in \mathcal{A}$},
$$
where $g: \mathbb{R} \rightarrow \mathbb{R}$ is a non-decreasing utility function.
\begin{definition}
    A pair $(\pi, \mu) \in \mathcal{A} \times \mathcal{P}_\mathcal{S}$ is a $G$-equilibrium (or a graphon equilibrium when $G$ is understood) of the mutual holding problem if $\mu \in \mathcal{P}_\mathcal{S}(\pi)$ and $J_{\pi, \mu}(\pi) = \sup_\beta J_{\pi, \mu}(\beta)$.
\end{definition}

\subsection{Explicit solution and well-posedness of the SDE}
\begin{proposition}
    For the given $\pi^*(t,u,x,v,y) := \mathbf{1}_{\{B^\mu(t,v,y) \geq 0\}}$, the limiting dynamics \eqref{E:eqn:53} take the simpler form:
    \begin{equation}
        \mathrm{d}X_t = B(t,U,X_t, \mu_t)\mathrm{d}t + \Sigma(t,U, X_t, \mu_t)\mathrm{d}W_t, \quad  \mu_t =\mathcal{L}(U,X_t)
        \label{E:eqn:6}
    \end{equation}
    with $B, \Sigma: [0,T] \times I \times \mathbb{R} \times \mathcal{P}^2 (I \times \mathbb{R}) \rightarrow \mathbb{R}$ being defined as
    \begin{align}
        B^\mu(t,u,x) & := \frac{1}{2}\Big(b(t,x,[G\mu](u)) + c(t, \mu)\Big)^+ - \Big(b(t,x,[G\mu](u)) + c(t, \mu)\Big)^-, \label{E:eqn:3} \\
        \Sigma^\mu(t,u,x) &:= \left(1-\frac{1}{2}\mathbf{1}_{\{B^\mu(t,u,x)\geq 0\}}\right) \sigma(t,x,[G\mu](u)), \label{E:eqn:4}
    \end{align}
    where $c: [0,T] \times \mathcal{P}^2(I \times \mathbb{R}) \rightarrow \mathbb{R}_+$ is the unique solution of the integral equation 
    \begin{equation}
        c(t, \mu) = \frac{1}{2} \int \Big(c(t,\mu)+b(t,y,[G\mu](v))\Big)^+ \mu(\mathrm{d}v,\mathrm{d}y).
        \label{E:eqn:52}
    \end{equation}
\end{proposition}

\begin{proof}
To verify \eqref{E:eqn:3}-\eqref{E:eqn:4}, plugging the expression for $\pi^*$ into \eqref{E:eqn:1}, we get 
\begin{equation}
    B^\mu(t,u,x) = \frac{b(t,x,[G\mu](u)) + \int (B^\mu)^+ (t,v,y)\mu(\mathrm{d}v, \mathrm{d}y)}{1+\mathbf{1}_{\{B^\mu(t,u,x) \geq 0\}}}.
    \label{E:eqn:5}
\end{equation}
Multiplying both sides by $\mathbf{1}_{\{B^\mu(t,u,x) \geq 0\}}$ gives
$$
2 (B^\mu)^+(t,u,x) = \mathbf{1}_{\{B^\mu(t,u,x) \geq 0\}} \left(b(t,x,[G\mu](u)) + \int (B^\mu)^+ (t,v,y)\mu(\mathrm{d}v, \mathrm{d}y)\right),
$$
which implies by integration with respect to $\mu$ and the fact that $\{(u,x): B^\mu(t,u,x) \geq 0\} = \{(u,x): b(t,x,[G\mu](u))+c(t,\mu) \geq 0\}$,
\begin{align*}
    & \int (B^\mu)^+ (t,v,y)\mu(\mathrm{d}v, \mathrm{d}y) \\
    & \quad = \frac{\int \Big(b(t,x,[G\mu](u))+c(t,\mu)\Big)^+\mu(\mathrm{d}u, \mathrm{d}x) - c(t,\mu)\int \mathbf{1}_{\{B^\mu(t,u,x)\geq 0\}}\mu(\mathrm{d}u, \mathrm{d}x)}{2-\int \mathbf{1}_{\{B^\mu(t,u,x) \geq 0\}}\mu(\mathrm{d}u, \mathrm{d}x)} \\
    & \quad = c(t, \mu),
\end{align*}
where we used the definition of $c$ for the last equality. Substituting this into \eqref{E:eqn:5}, we get \eqref{E:eqn:3}, and this leads to the graphon equilibrium mean-field SDE \eqref{E:eqn:6}. We postpone the justification of the unique existence of $c$ in the general case to the next lemma.
\end{proof}

\begin{lemma}
    Let $\varphi: \mathbb{R} \times \mathcal{M}^2(\mathbb{R}) \rightarrow \mathbb{R}$ be a function such that $\varphi(\cdot, [G \mu] (\cdot))$ is $\mu$-integrable for any $\mu \in \mathcal{P}([0,1] \times \mathbb{R})$. Then, there exists a unique
    $$
    c^\varphi: \mathcal{P}([0,1] \times \mathbb{R}) \rightarrow \mathbb{R}_+ \text{ s.t. } c^\varphi = \frac{1}{2} \int \big(c^\varphi + \varphi(x, [G \mu](u))\big)^+ \mu(\mathrm{d}u, \mathrm{d}x).
    $$
    Moreover, $c^\varphi \leq \int \varphi^+ (x, [G \mu](u)) \mu(\mathrm{d}u, \mathrm{d}x)$.
\end{lemma}

\begin{proof}
Similar to the proof of Lemma 4.8 in \cite{DjeteTouzi2024}, and hence omitted.     
\end{proof}

Before stating the main result of this section, i.e. Theorem \ref{E:thm:1}, we impose some conditions on the coefficients of the SDE defining the dynamics of the idiosyncratic risk process.

\begin{assumption} The coefficients $b$, $\sigma$ are bounded continuous functions, Lipschitz in $(x, m)$ uniformly in $t$ in the sense that there exists $C > 0$ s.t. for any $x, x^\prime \in \mathbb{R}$ and $m, m^\prime \in \mathcal{M}^2(\mathbb{R})$:
    \begin{equation}
        \sup_{t \in [0,T]} \left|(b, \sigma)(t,x,m) - (b, \sigma)(t,x^\prime, m^\prime)\right| \leq C\left(\left|x-x^\prime\right| + \mathcal{WOP}_2(m, m^\prime)\right),
        \label{E:eqn:54}
    \end{equation}
    and $\sigma$ is bounded from below away from zero uniformly in $(t,x,m)$. Moreover,
    \begin{itemize}
        \item[\textup{(i)}] either for each $m \in \mathcal{M}(\mathbb{R})$, $b(t,x,m) < 0$, for a.e. $(t,x) \in [0, T] \times \mathbb{R}$, 

        \item[\textup{(ii)}] or for any $\mu \in \mathcal{P}([0,1] \times \mathbb{R})$ and $\eta \in [0, \infty)$, and for Lebesgue a.e. $t \in [0,T]$, the Borel set
        \begin{align*}
            \ell (t,\mu, \eta) := & \big\{ (u,x) \in I \times \mathbb{R}: (u^\prime, x^\prime) \mapsto \mathbf{1}_{\{b(t, x^\prime, [G\mu](u^\prime)) + \eta \geq 0\}} \text{ continuous at $(t,x,[G\mu](u), \eta)$}\big\}
        \end{align*}
        has full Lebesgue measure, i.e. its complement is Lebesgue-negligible.
    \end{itemize}
    \label{E:assump:1}
\end{assumption}

\begin{remark}
    1. If $\mu, \mu^\prime \in \mathcal{P}^2(\mathbb{R})$, then $\mathcal{WOP}_2 (\mu, \mu^\prime) = \mathcal{W}_2 (\mu, \mu^\prime)$, or more generally, $\mu, \mu^\prime \in \mathcal{M}^2(\mathbb{R})$ with the same total mass $m_0$, then $\mathcal{WOP}_2 (\mu, \mu^\prime) = m_0 \mathcal{W}_2(\overline{\mu}, \overline{\mu}^\prime)$. 

    \vspace{1mm}
    
    \noindent 2. The condition \eqref{E:eqn:54} is natural and satisfied by many examples. For example, let 
    $$
    \mathcal{M}^2(\mathbb{R}) \ni \mu \mapsto b (\mu) := \int_\mathbb{R} \big[m_\mu (x -x_0) +x_0\big] \overline{\mu} (\mathrm{d}x) \in \mathbb{R}, 
    $$
    where $m_\mu$ denotes the total mass of $\mu$ as usual. For any $\mu, \nu \in \mathcal{M}^2(\mathbb{R})$, we calculate
    \begin{align*}
        |b(\mu) - b(\nu)| & = \left|\int_\mathbb{R} \big[m_\mu (x-x_0)+x_0\big] \overline{\mu} (\mathrm{d}x) - \int_\mathbb{R} \big[m_\nu (x-x_0) + x_0\big] \overline{\nu} (\mathrm{d}x) \right|  \\
        & = \left|\int_\mathbb{R} x\left( T_{m_\mu}\# \overline{\mu}\right) (\mathrm{d}x) - \int_\mathbb{R} x \left(T_{m_\nu}\# \overline{\nu}\right) (\mathrm{d}x) \right| \\
        & \leq \mathcal{W}_2 (T_{m_\mu}\# \overline{\mu}, T_{m_\nu}\# \overline{\nu} ) \\
        & \leq \mathcal{WOP}_2 (T_{m_\mu}\# \overline{\mu}, T_{m_\nu}\# \overline{\nu} ),
    \end{align*}
    where we used of the definition of $\mathcal{WOP}_2$ metric for the last inequality. 
    
    Another simple example is that $b$ depends on the measure argument totally through its total mass, i.e. $b (\mu) = m_\mu$. Then,
    \begin{equation*}
        |b(\mu) -b(\nu)| = |m_\mu - m_\nu| \leq \mathcal{WOP}_2(\mu,\nu), \text{ for any $\mu,\nu \in \mathcal{M}^2(\mathbb{R})$}.
    \end{equation*}
\end{remark}

\begin{theorem}
    Let Assumption \ref{E:assump:1} hold true. In addition, suppose that $G$ is continuous on $[0,1]^2$, and $\mu_0 \in \mathcal{P}^p_{\textup{Unif}}([0,1] \times \mathbb{R})$ with $p >2$. Then, we have
    \begin{itemize}
        \item[\textup{(i)}] The mean-field SDE \eqref{E:eqn:6} has at least one square-integrable weak solution.

        \item[\textup{(ii)}] If $g$ is non-decreasing and Lipschitz, any weak solution distribution $\mu$ of \eqref{E:eqn:6} induces a $G$-equilibrium $(\pi^*, \mu)$ of the mutual holding problem, where $\pi^* (t,u,x,v,y) := \mathbf{1}_{\{B^\mu(t,v,y)\geq 0\}}$.
    \end{itemize}
    \label{E:thm:1}
\end{theorem}

\begin{remark}
    In \cite{DjeteGuo2023}, Guo et al. proved in their setting that there exists a unique strong solution to the equilibrium McKean-Vlasov SDE when $b$ has a constant sign and $b, \sigma$ are independent of the distribution variable. However, in our setting the heterogeneous interaction between agents, defined as the positive finite measures, depends on the measure argument, hence it is not clear that  their techniques is applicable to the current context.
\end{remark}

\subsection{Propagation of Chaos}
Let $(\Omega, \mathcal{F}, \mathbb{F}, \mathbb{P})$ be a filtered probability space supporting a sequence of independent Brownian motions $(W^i)_{1 \leq i \leq N}$, and denote by $\mathbf{W} := (W^1, \cdots, W^N)$. Let $\Gamma := (\gamma^{i,j})_{1 \leq i,j \leq N}$ be a $[0,1]^{N \times N}$-valued $\mathbb{F}$-predictable process, and $\mathbf{X}:=(X^1, \cdots, X^N)$ the solution of the SDE:
\begin{equation*}
    \mathrm{d} X_t^i = \frac{1}{N} \sum_{j=1}^N \gamma_t^{i,j}\mathrm{d}X_t^j - \frac{1}{N} \sum_{j=1}^N \gamma_t^{j,i}\mathrm{d}X_t^i + b(t,X_t^i, M^{i,N}_t) \mathrm{d}t + \sigma(t,X_t^i, M^{i,N}_t)\mathrm{d}W_t^i
\end{equation*}
with $M^{i,N}_t$ defined in \eqref{E:eqn:16}. Rewriting in a vector form gives:
$$
M(\Gamma_t) \mathrm{d}\mathbf{X}_t = \vec{b} (t, \mathbf{X}_t, \mathbf{M}_t^N) \mathrm{d}t + \text{diag}\left[\vec{\sigma} (t, \mathbf{X}_t, \mathbf{M}^N_t)\right] \mathrm{d}\mathbf{W}_t,
$$
where $M(\Gamma_t)$ is a matrix depending on $\Gamma_t$, and for $\varphi = b \text{ or } \sigma$, we denote by $\vec{\varphi}(t, \mathbf{X}_t, \mathbf{M}_t^N)$ the vector in $\mathbb{R}^N$ with $i$-th entry $\varphi(t, X^i_t, M^{i,N}_t)$, and $\text{diag}[\vec{\varphi}]$ is the diagonal matrix with diagonal elements defined by the entries of the vector $\vec{\varphi}$. $M(\Gamma_t)$ is a strictly diagonally dominant matrix and it is therefore invertible. Consequently, $\mathbf{X}$ is an It\^o process defined by the drift and diffusion coefficients $B=(B^i)_{1 \leq i \leq N}$ and $\Sigma=(\Sigma^{i,j})_{1\leq i,j \leq N}$:
\begin{align}
    \mathrm{d}\mathbf{X}_t = B_t \mathrm{d}t + \Sigma_t \mathrm{d}\mathbf{W}_t, \; & \text{with } B_t = M(\Gamma_t)^{-1} \vec{b} (t, \mathbf{X}_t, \mathbf{M}^N_t) =: \mathbf{B} (t, \Gamma_t, \mathbf{u}, \mathbf{X}_t), \nonumber \\
    & \text{and } \Sigma_t = M(\Gamma_t)^{-1} \text{diag}\left[\vec{\sigma} (t, \mathbf{X}_t, \mathbf{M}^N_t)\right] =:  \mathbf{\Sigma} (t, \Gamma_t, \mathbf{u}, \mathbf{X}_t). \label{E:eqn:17}
\end{align}
For later use, we isolate the equations defining the coefficients $(B^i)_{1 \leq i \leq N}$ and $(\Sigma^{i,j})_{1 \leq i,j \leq N}$:
\begin{align}
    & B^i_t = \frac{1}{N} \sum_{j=1}^N \gamma_t^{i,j} B_t^j - \frac{1}{N} \sum_{j=1}^N \gamma_t^{j,i} B_t^i + b(t, X^i_t, M^{i,N}_t), \label{E:eqn:23}\\
    & \Sigma^{i,q}_t = \frac{1}{N} \sum_{j=1}^N \gamma^{i,j}_t \Sigma^{j,q}_t - \frac{1}{N} \sum_{j=1}^N \gamma^{j,i}_t \Sigma^{i,q}_t + \sigma(t, X_t^i, M^{i,N}_t) \mathbf{1}_{q=i} 
    \label{E:eqn:24}.
\end{align}

\subsubsection{Deviating player} 
For any $[0, 1]^N$-valued $\mathbb{F}$-predictable process $\beta^i := (\beta^{i,1}, \cdots, \beta^{i,N})^\top$, we introduce the deviated matrix strategy defined by substituting $(\beta^i)^\top$ to the $i$-th line of $\Gamma$:
$$
\Gamma^{-i} (\beta^i) := \left(\left(\gamma^{1,\cdot}\right)^\top, \cdots, \left(\gamma^{i-1, \cdot}\right)^\top, \beta^i, \left(\gamma^{i+1, \cdot}\right)^\top, \cdots, \left(\gamma^{N, \cdot}\right)^\top\right)^\top \text{ where } \gamma^{i, \cdot} := (\gamma^{i,j})_{1 \leq j \leq N}.
$$
Similar to the mean-field setting, we introduce an equivalent probability measure $\mathbb{P}^i_{\Gamma, \beta^i}$ defined by the Radon-Nykodim density:
$$
\frac{\mathrm{d} \mathbb{P}^i_{\Gamma, \beta^i}}{\mathrm{d}\mathbb{P}} := Z^i_T := \exp{\left(\int_0^T \psi^i_t \mathrm{d}W^i_t - \frac{1}{2}|\psi^i_t|^2\mathrm{d}t\right)}, \quad \text{on } \mathcal{F}_T
$$
with
$$
\psi^i_t := \frac{\mathbf{B}^i (t, \Gamma^{-i}_t(\beta^i_t), \mathbf{u},\mathbf{X}_t) - \mathbf{B}^i(t, \Gamma_t, \mathbf{u}, \mathbf{X}_t)}{\mathbf{\Sigma}^{i,i}(t, \Gamma^{-i}_t (\beta^i_t), \mathbf{u}, \mathbf{X}_t)},
$$
and we define $\widehat{W}^{i,i}_\cdot := W^i_\cdot - \int_0^\cdot \psi^i_s \mathrm{d}s$. Then, it follows from the Girsanov theorem that the process $\widehat{\mathbf{W}}^i := (W^1, \cdots, W^{i-1}, \widehat{W}^{i,i}, W^{i+1}, \cdots, W^N)$ is a $\mathbb{P}^i_{\Gamma, \beta^i}$-Brownian motion.

We now introduce $\mathbf{X}^i := (X^{i,1}, \cdots, X^{i,N})$ defined by the stochastic differential equation:
$$
\mathrm{d}\mathbf{X}_t^i = \mathbf{B}(t, \Gamma_t^{-i}(\beta^i_t), \mathbf{u}, \mathbf{X}^i_t) \mathrm{d}t + \mathbf{\Sigma}(t, \Gamma_t^{-i}(\beta^i_t), \mathbf{u}, \mathbf{X}^i_t) \mathrm{d}\widehat{\mathbf{W}}^i_t,
$$
where the functions $\mathbf{B}$ and $\mathbf{\Sigma}$ are defined in \eqref{E:eqn:17}. Given $\Gamma$, notice that $\mathbf{X}^i$ is well-defined since the maps $\mathbf{B}$ and $\mathbf{\Sigma}$ are Lipschitz in $\mathbf{x}$ (see Lemma \ref{E:lemma:2} below).

Consequently, we may define the reward of the $i$-th deviated player from using the strategy $\beta^i$, given that the remaining agents stick to their strategies $\gamma^{k,\cdot}$, $k \neq i$, by
$$
J_i (\Gamma^{-i}(\beta^i)) := \mathbb{E}^{\mathbb{P}^i_{\Gamma, \beta^i}} \big[g(X_T^{i,i})\big], \quad \text{for } i =1, \ldots, N.
$$

\begin{definition}
    (Approximate Nash equilibrium). For $\boldsymbol{\epsilon}:=(\epsilon^1, \cdots, \epsilon^N) \in [0, \infty)^N$, we say that $\Gamma$ is an $\boldsymbol{\epsilon}$-Nash equilibrium if:
    $$
    J_i (\Gamma) \geq \sup_\beta J_i(\Gamma^{-i}(\beta)) - \epsilon^i, \quad \text{for all $i \in \{1, \cdots, N\}$}.
    $$
    \label{E:def:1}
\end{definition}

\begin{assumption}
    For some $p > 2$, the initial law $\mu_0 \in \mathcal{P}_{\textup{Unif}}^p ([0,1] \times \mathbb{R})$ admits $\mu_0(\mathrm{d}u, \mathrm{d}x) = \mu_{u,0}(\mathrm{d}x)\mathrm{d}u$, uniquely defined up to a.e. equality, and satisfies
    $
    \sup_{u \in I} \int_\mathbb{R} |x|^p \mu_{u,0}(\mathrm{d}x) < \infty.
    $
    \label{E:assump:2}
\end{assumption}

\subsubsection{Construction of approximate Nash equilibrium} 
Suppose $\mathbf{u} := (u^1, \cdots, u^N) \in I^N_1 \times \cdots \times I^N_N$ represents the labels of $N$ players, and define for $\mathbf{x} := (x^1, \cdots, x^N) \in \mathbb{R}^N$:
$$
m^N := \frac{1}{N} \sum_{j=1}^N \delta_{(u^j,x^j)}, \quad \text{and } m^{i,N} := [G^N m^N](u^i) = \frac{1}{N}\sum_{j=1}^N \xi^N_{ij} \delta_{x^j}.
$$
Denote $\overline{\mathbf{x}} := (\mathbf{u}, \mathbf{x})$ and define for $i,j = 1, \cdots, N$:
\begin{equation}
\left\{
\begin{aligned}
    & \pi (t,u^j,x^j, m^N) := \pi^j (t, \overline{\mathbf{x}}) := \mathbf{1}_{\{B(t,u^j,x^j, m^N) \geq 0\}} \\
    & B^i_t (\overline{\mathbf{x}}) = \frac{\frac{1}{N}\sum_{k=1}^N A^k_t( \overline{\mathbf{x}}) b(t,x^k,m^{k,N}) + b(t,x^i, m^{i,N})}{1+\pi^i(t, \overline{\mathbf{x}})} \\
    & \Sigma^{i,j}_t (\overline{\mathbf{x}}) = \frac{\frac{1}{N}A^j_t(\overline{\mathbf{x}}) \sigma(t,x^j,m^{j,N}) + \sigma(t,x^i, m^{i,N}) \mathbf{1}_{i=j}}{1+\pi^i(t,\overline{\mathbf{x}})}
\end{aligned}
\right.
\label{E:eqn:8}
\end{equation}
with $B$ given in \eqref{E:eqn:3}, where $A^j_t (\overline{\mathbf{x}})$ is defined as follows:
$$
A^j_t(\overline{\mathbf{x}}) := \frac{\frac{\pi^j(t,\overline{\mathbf{x}})}{1+\pi^j(t, \overline{\mathbf{x}})}}{1-\frac{1}{N}\sum_{k=1}^N\frac{\pi^k(t, \overline{\mathbf{x}})}{1+\pi^k(t, \overline{\mathbf{x}})}}.
$$

By Theorem 2.6.1 in \cite{Krylov1980}, there exists a weak solution $\mathbf{X} := (X^1, \cdots, X^N)$ to the stochastic differential equation:
\begin{align*}
    \mathrm{d} X_t^i & = \frac{1}{N} \sum_{j=1}^N \pi_t^{i,j}\mathrm{d}X_t^j - \frac{1}{N} \sum_{j=1}^N \pi_t^{j,i}\mathrm{d}X_t^i + b(t,X_t^i, M^{i,N}_t) \mathrm{d}t + \sigma(t,X_t^i, M^{i,N}_t)\mathrm{d}W_t^i \\
    & = B^i(t, \mathbf{u}, \mathbf{X}_t, \mu^N_t) \mathrm{d}t + \sum_{j=1}^N \Sigma^{i,j}(t, \mathbf{u}, \mathbf{X}_t, \mu^N_t) \mathrm{d}W^j_t
\end{align*}
with $\mathcal{L} (X^i_0) = \mu_{u^i,0}$, where the last equality follows by direct verification and we used the notation:
$$
M^{i,N}_t := [G^N \mu^N_t](u^i) = \frac{1}{N}\sum_{j=1}^N \xi^N_{ij} \delta_{X^j_t}, \quad \text{and } \mu^N_t := \frac{1}{N} \sum_{j=1}^N \delta_{(u^j, X_t^j)},
$$
and
\begin{equation}
    \pi^{i,j}_t := \pi^j_t := \pi(t,u^j,X^j_t, \mu^N_t), \quad \Pi^N := (\pi^{i,j})_{1 \leq i, j \leq N} \text{ for all $N \geq 1$}.
    \label{E:eqn:18}
\end{equation}

With the fixed $\mathbf{u}^N:=(u^1, \cdots, u^N)$, we finally define the nonnegative number 
\begin{equation}
    \epsilon^{i,N} (\mathbf{u}^N) := \sup_{\beta} J_i((\Pi^N)^{-i}(\beta)) - J_i(\Pi^N) \geq 0.
    \label{E:eqn:50}
\end{equation}
By Definition \ref{E:def:1}, $\Pi^N$ is an $\boldsymbol{\epsilon}^N (\mathbf{u}^N) := (\epsilon^{1,N}(\mathbf{u}^N), \cdots, \epsilon^{N,N}(\mathbf{u}^N))$-Nash equilibrium. This definition makes sense only if we prespecify a version of the disintegration $u \mapsto \mu_{u,0}$, and otherwise we should understand $\boldsymbol{\epsilon}^N(\mathbf{u}^N)$ to be uniquely defined only up to a.e. equality. When $\mathbf{u}^N$ is known, we simply write $\boldsymbol{\epsilon}^N := (\epsilon^{1,N}, \cdots, \epsilon^{N,N})$.

For any $\mu \in \mathcal{P}_\mathcal{S}(\pi)$, we introduce the set 
$$
\mathcal{S}(\pi, \mu) := \{\mathbb{P}^\beta_{\pi, \mu} \circ (U, X)^{-1}, \; \beta \in \mathcal{A}\},
$$
and the sequence of probability distributions $(\mathrm{P}^N)_{N \in \mathbb{N}^*} \subset \mathcal{P} (\widehat{\Omega} \times \mathcal{P}(\widehat{\Omega}))$
$$
\mathrm{P}^{N} := \frac{1}{N} \sum_{i=1}^N \mathbb{P}^{i}_{\Pi, \beta^i} \circ \big(u^i, X^{i,i}, \mu^{i,N}\big)^{-1}.
$$
We denote by $(\widehat{U}, \widehat{X}, \widehat{\mu})$ the canonical elements on $\widehat{\Omega} \times \mathcal{P}(\widehat{\Omega})$.

\begin{theorem}
    Let Assumptions \ref{E:assump:1} and \ref{E:assump:2} hold true. Moreover, suppose that $G$ is continuous on $[0,1]^2$  and that $G^N$ is a step graphon defined in \eqref{E:eqn:21} satisfying either of the following conditions
    \begin{itemize}
        \item (Sampling kernel) $G^N(u^i, u^j) = G(u^i, u^j)$, for $i,j=1, \cdots,N$,

        \item (General kernel) $N^2 \big\|G^N -G\big\|_\square \rightarrow 0$.
    \end{itemize}
    
    \noindent Then we have that
    \begin{itemize}
        \item[\textup{(i)}] the sequence $(\mathrm{P}^N)_{N \in \mathbb{N}^*}$ is relatively compact in $\mathcal{W}_2$, and for any limit point $\mathrm{P}^\infty$
        $$
        \widehat{\mu} (\omega) \in \mathcal{P}_\mathcal{S}(\pi) \text{ and } \mathcal{L}^{\mathrm{P}^\infty} (\widehat{U}, \widehat{X} | \widehat{\mu}(\omega)) \in \mathcal{S}(\pi, \widehat{\mu}(\omega)), \quad  \mathrm{P}^\infty-a.s. \; \omega.
        $$

        \item[\textup{(ii)}] (Approximate Nash equilibrium) $\frac{1}{N} \sum_{i=1}^N \epsilon^{i,N} \rightarrow 0$.
    \end{itemize}
    \label{E:thm:2}
\end{theorem}

\begin{proof}
    The proof of (i) is postponed to Section \ref{E:sec:1}. To prove (ii), for all $i \in \{1, \cdots, N\}$, let $\beta^i \in \mathbb{R}^N$ satisfy
    $
    J_i((\Pi^N)^{-i}(\beta^i)) - J_i(\Pi^N) \geq \epsilon^{i,N} - 2^{-N}.
    $
    By the construction in \eqref{E:eqn:50}, $\frac{1}{N}\sum_{i=1}^N \epsilon^{i,N} \geq 0$ for all $N \geq 1$, and
    \begin{align*}
        & \limsup_{N \rightarrow \infty} \frac{1}{N} \sum_{i=1}^N \epsilon^{i,N} \leq \limsup_{N \rightarrow \infty} \frac{1}{N} \sum_{i=1}^N J_i\big((\Pi^N)^{-i}(\beta^i)\big) - J_i\big(\Pi^N\big) \\
        &  \qquad = \limsup_{N \rightarrow \infty} \frac{1}{N} \sum_{i=1}^N \mathbb{E}^{\mathbb{P}^i_{\Pi, \beta^i}} \left[g(X^{i,i}_T)\right] - \mathbb{E}^\mathbb{P} \left[\big\langle g, \mu^N_T\big\rangle\right] = \limsup_{N \rightarrow \infty} \frac{1}{N} \sum_{i=1}^N \mathbb{E}^{\mathbb{P}^i_{\Pi, \beta^i}} \left[g(X^{i,i}_T) - \big\langle g, \mu^{i,N}_T\big\rangle\right] \\
        &  \qquad = \mathbb{E}^{\mathrm{P}^\infty}\left[g(\widehat{X}_T) - \int_{\widehat{\Omega}} g(x_T) \widehat{\mu}(\mathrm{d}u, \mathrm{d}x)\right] = \mathbb{E}^{\mathrm{P}^\infty}\left[\mathbb{E}^{\mathrm{P}^\infty} \left[g(\widehat{X}_T) \big| \widehat{\mu}\right] - \int_{\widehat{\Omega}} g(x_T) \widehat{\mu}(\mathrm{d}u, \mathrm{d}x)\right] \\
        &  \qquad \leq 0,
    \end{align*}
    where the last two lines are guaranteed by (i), and then we used the fact that any distribution $\mu \in \mathcal{P}_\mathcal{S}(\pi)$ with the corresponding control $\pi$ is a graphon equilibrium of the mutual holding problem. We conclude the proof by noticing that
    $$
    0 \leq \liminf_{N \rightarrow \infty} \frac{1}{N} \sum_{i=1}^N  \epsilon^{i,N} \leq \limsup_{N \rightarrow \infty} \frac{1}{N}\sum_{i=1}^N \epsilon^{i,N} \leq 0.
    $$
\end{proof}

\begin{remark}
    The continuity assumption on the map $(u,v) \mapsto G(u,v)$ is both necessary for Theorem \ref{E:thm:1} and Theorem \ref{E:thm:2}. When dealing with weak convergence of joint distributions with the common first marginal (or more generally strong convergence in first marginal), we can relax from continuous functions to Carath{\'e}odory function, see for example, Lemma A.10 in \cite{BeigJour2023}. 
    
    However, in the proof of Theorem \ref{E:thm:2} (i), the sequence $(\frac{1}{N}\sum_{i=1}^N \delta_{u^i})_{N \geq 1}$ is neither the same as the limiting marginal $\text{Unif}([0,1])$ nor strongly converges to it when $N \rightarrow \infty$. Consequently, if one relaxes from continuous kernel to merely measurable kernel, we shall encounter difficulties in proving the continuity of $c(t,\mu)$ w.r.t. $\mu$, which is crucial to ensure the usage of the Portmanteau theorem in the proof of Theorem \ref{E:thm:1}.
    For any $\mu,\nu \in \mathcal{P}^2_{\textup{Unif}} (I\times \mathbb{R})$ with disintegrations $\mu(\mathrm{d}u, \mathrm{d}x) = \mu_u(\mathrm{d}x)\mathrm{d}u$ and $\nu(\mathrm{d}u, \mathrm{d}x) = \nu_u(\mathrm{d}x)\mathrm{d}u$, consider 
    \begin{align*}
        & |c(t,\mu) - c(t,\nu)| \\
        & = \frac{1}{2}\left|\int_I \left[\int_{\mathbb{R}} \big(c(t,\mu) + b(t,x,[G \mu](u))\big)^+ \mu_u(\mathrm{d}x) -\int_{\mathbb{R}} \big(c(t,\nu) + b(t,x,[G \nu](u))\big)^+ \nu_u(\mathrm{d}x)\right]\mathrm{d}u \right| \\
        & \leq \frac{C_b}{2}\int_I \left[\inf_{\pi_u \in \Pi(\mu_u, \nu_u)} \int_{\mathbb{R}^2} \big|x-x^\prime\big| \pi_u(\mathrm{d}x, \mathrm{d}x^\prime) + \mathcal{WOP}_2 ([G\mu](u), [G \nu](u)) \right]\mathrm{d}u + \frac{1}{2} |c(t,\mu) - c(t,\nu)|.
    \end{align*}
    Consequently we can derive the upper bound:
    $$
    |c(t,\mu) - c(t, \nu)| \leq C \left(\int_I \mathcal{W}_2^2 (\mu_u, \nu_u)\mathrm{d}u\right)^{1/2},
    $$
    where $C_b >0$ only depends on $b$. However, noticing that 
    $
    \mathcal{W}_2^2 (\mu, \nu) \leq \int_I \mathcal{W}_2^2 (\mu_u, \nu_u) \mathrm{d}u,
    $
  this implies that $\mathcal{W}_2(\mu, \nu) = 0$ does not lead to $c(t, \mu) = c(t,\nu)$. Another potential approach is
    \begin{align*}
        |c(t,\mu)  -c(t,\nu)| \leq  C_b \inf_{\pi \in \Pi(\mu,\nu)} \int_{(I \times \mathbb{R})^2} \left(|x-x^\prime| + \mathcal{WOP}_2 ([G \mu](u), [G \nu](u^\prime))\right) \pi(\mathrm{d}u, \mathrm{d}x, \mathrm{d}u^\prime, \mathrm{d}x^\prime).
    \end{align*}
    It is also clear that without the continuity of $G$, it is not possible to deal with $ \mathcal{WOP}_2([G \mu](u), [G \nu](u^\prime)) $.
\end{remark}

\begin{remark}
    Our approximate equilibrium deals with either the case of sampling kernel  or general kernel with $N^2 \|G^N -G\|_\square \rightarrow 0$. If one assumes instead that the graph sequence converges in the cut-norm (resp. $L^1$-norm or the strong operator topology), i.e. $\|G^N - G\|_\square \rightarrow 0$ (resp. $\|G^N-G\|_{L^1([0,1]^2)} := \int_0^1 \int_0^1 |G^N(u,v) - G(u,v)| \mathrm{d}u \mathrm{d}v \rightarrow 0$ or $\|G^N -G\|_{\infty \rightarrow 1} \rightarrow 0$, with operator norm $\|\cdot\|_{\infty \rightarrow 1}$ defined by:
    $$
    \|G\|_{\infty \rightarrow 1} := \sup_{\|g\|_\infty \leq 1} \|G g\|_{L^1([0,1])} = \sup_{\|g\|_\infty \leq 1} \int_I \left|\int_I G(u,v)g(v)\mathrm{d}v\right|\mathrm{d}u
    $$
    ), the arbitrary closeness of $\frac{1}{N} \sum_{j=1}^N \xi^N_{ij}$ and $\frac{1}{N} \sum_{j=1}^N G(u^i,u^j)$ is not guaranteed due to the lack of pointwise control, which is key to the proof Lemma \ref{E:lemma:4}.
\end{remark}

\begin{remark}
    Notice that in \cite[Theorem 3]{LackerSoret2023} the authors are able to prove a stronger convergence result in the case of continuous kernel, i.e.
    $$
    \lim_{N \rightarrow \infty} \esssup_{\mathbf{u}^N \in I^N_1 \times \cdots \times I^N_N} \frac{1}{N} \sum_{i=1}^N \epsilon^{i,N} (\mathbf{u}^N) = 0,
    $$
     We are unable to prove this type of convergence result, due to the facts that the dependence of our approximation errors $(\epsilon^{i,N})_{1 \leq i \leq N}$ on the interaction term $\mathbf{u}^N$ is less explicit. In particular, their interaction terms are included only in the cost functions, not in the coefficients of the dynamics.
\end{remark}

\section{Well-posedness and verification theorem} \label{sec:wellpose}
\subsection{Technical lemmas} \label{subsec:tech} 
We denote by $\mathcal{S} := \mathcal{S}_1 \cup \mathcal{S}_2$ a set of probability measures on $[0,1] \times \mathbb{R}$, where $\mathcal{S}_1$ and $\mathcal{S}_2$ are defined by
\begin{equation}
\begin{aligned}
    \mathcal{S}_1 := \mathcal{P}_{\textup{Unif}}^2 ([0,1] \times \mathbb{R}), \quad \text{and } \mathcal{S}_2 := \bigg\{\frac{1}{k}\sum_{i=1}^k \delta_{(u^i,x^i)}: \; \text{for some $k \geq 1$ with $u^i \in I^k_i$ and $x^i \in \mathbb{R}$ }\bigg\}.
\end{aligned}
\label{E:eqn:51}
\end{equation}
Next, we prove the joint continuity of $(u,\mu) \mapsto [G\mu](u)$ in $\mathcal{WOP}_2$ as well as $\mu \mapsto c(t, \mu)$ for a.e. $t \in [0,T]$. Such continuity is needed to ensure the validity of the Portmanteau theorem in both the proof of well-posedness and propagation of chaos. In the following, we suppose that $G$ is a continuous function defined on $[0,1]^2$. We first recall a slightly extended version of the Arzel\`a–Ascoli theorem.
\begin{lemma}[Exercise 16 in Chapter 7, \cite{Rudin1976}]
    Suppose that $\{f_n\}$ is an equicontinuous sequence of functions on a compact set $K$, and $\{f_n\}$ converges pointwise on $K$. Then, $\{f_n\}$ converges uniformly on $K$.
    \label{E:lemma:6}
\end{lemma}

\begin{lemma}
    Let $\mu, \mu^k \in \mathcal{S}$ satisfy $\mathcal{W}_2 (\mu^k, \mu) \rightarrow 0$ with $\mu \in \mathcal{P}^2_{\textup{Unif}}([0,1]\times \mathbb{R})$. In addition, suppose that $\widehat{u}^k \rightarrow \widehat{u}$ with $\widehat{u}^k, \widehat{u} \in [0,1]$. Then, we have:
    \begin{itemize}
        \item[\textup{(a)}] $\lim_{k \rightarrow \infty} \mathcal{WOP}_2 \big([G \mu^k](\widehat{u}^k), [G \mu](\widehat{u})\big) = 0$.

        \item[\textup{(b)}] $c(t, \mu^k) \rightarrow c(t, \mu)$, for a.e. $t \in [0,T]$. 
    \end{itemize}
    \label{E:lemma:7}
\end{lemma}
\begin{proof}
    (a) We provide the proof for the case of $\mu^k \in \mathcal{S}_2$, $\forall k \geq 1$. For $\mu^k \in \mathcal{S}_1$, the proof is similar and indeed simpler. We start by considering
    \begin{align*}
        & \mathcal{WOP}_2^2 \big([G \mu^k](\widehat{u}^k), [G \mu](\widehat{u})\big) = \bigg(\int_I G(\widehat{u},v)\mathrm{d}v - \frac{1}{k}\sum_{i=1}^k G(\widehat{u}^k, u^i)\bigg)^2 \\
        & \quad + \bigg(\int_I G(\widehat{u},v)\mathrm{d}v\bigg)\bigg(\frac{1}{k}\sum_{i=1}^k G(\widehat{u}^k, u^i)\bigg) \mathcal{W}_2^2 \Big( \overline{[G \mu^k](\widehat{u}^k)}, \overline{[G \mu](\widehat{u})}\Big) \\
        & \quad + \bigg(\int_I G(\widehat{u},v)\mathrm{d}v - \frac{1}{k}\sum_{i=1}^k G(\widehat{u}^k, u^i)\bigg) \bigg[\int_{I \times \mathbb{R}} G(\widehat{u},v)x^2 \mu(\mathrm{d}v, \mathrm{d}x) - \int_{I \times \mathbb{R}} G(\widehat{u}^k, v)x^2 \mu^k(\mathrm{d}v, \mathrm{d}x)\bigg] \\
        & =: \text{I}_k + \text{II}_k + \text{III}_k.
    \end{align*}
    
    For $\text{I}_k$ and $\text{III}_k$, using the uniform continuity of $G$ as well as the definition of Riemann integral, we have for any $\epsilon > 0$, there exists some $K_1$ such that for any $k > K_1$ 
    \begin{align}
        \bigg|\int_I G(\widehat{u},v)\mathrm{d}v - \frac{1}{k}\sum_{i=1}^k G(\widehat{u}^k, u^i)\bigg| & \leq \bigg|\int_I G(\widehat{u},v)\mathrm{d}v - \frac{1}{k}\sum_{i=1}^k G(\widehat{u}, u^i)\bigg| + \frac{1}{k}\sum_{i=1}^k \big|G(\widehat{u}, u^i) - G(\widehat{u}^k,u^i)\big| \nonumber \nonumber \\
        & \leq \epsilon. \label{E:eqn:39}
    \end{align}
    Since $\mathcal{W}_2$-convergence implies the second-order moment convergence, from $\mathcal{W}_2(\mu^k, \mu) \rightarrow 0$, we have $\sup_k \int_{I \times \mathbb{R}} (u^2+x^2) \mu^k (\mathrm{d}u, \mathrm{d}x) < \infty.$
    Then, for any $k > K_1$, using \eqref{E:eqn:39} gives
    \begin{equation}
        \text{III}_k \leq C \bigg|\int_I G(\widehat{u},v)\mathrm{d}v - \frac{1}{k}\sum_{i=1}^k G(\widehat{u}^k, u^i)\bigg|  \leq C \epsilon.
        \label{E:eqn:48}
    \end{equation}

    Next, we claim that 
    \begin{equation}
        \text{II}_k := \bigg(\int_I G(\widehat{u},v)\mathrm{d}v\bigg)\bigg(\frac{1}{k}\sum_{i=1}^k G(\widehat{u}^k, u^i)\bigg) \mathcal{W}_2^2 \Big( \overline{[G \mu^k](\widehat{u}^k)}, \overline{[G \mu](\widehat{u})}\Big) \longrightarrow 0,
        \label{E:eqn:49}
    \end{equation}
    which ensures that for the previously specified $\epsilon$, there exists a $K_2 > 0$ such that for any $k > K_2$, we have $\text{II}_k \leq \epsilon$. Combined with \eqref{E:eqn:39}-\eqref{E:eqn:48}, we have for $k > \max\{K_1, K_2\}$ that
    $$
    \mathcal{WOP}_2^2 \big([G \mu^k](\widehat{u}^k), [G \mu](\widehat{u})\big) \leq C \epsilon \quad \text{with $C$ independent of $k$}.
    $$
    
    We complete the first part proof by verifying the claim stated in \eqref{E:eqn:49}. We firstly assume that $\int_I G(\widehat{u},v)\mathrm{d}v \neq 0$, and notice that
    \begin{align}
        & \lim_{R \rightarrow \infty} \limsup_{k \rightarrow \infty}\frac{1}{\frac{1}{k}\sum_{i=1}^k G(\widehat{u}^k, u^i)} \int_{I \times \mathbb{R}} G(\widehat{u}^k,v)x^2 \mathbf{1}_{\{|x| \geq R\}} \mu^k(\mathrm{d}v, \mathrm{d}x)  \nonumber \\
        & \leq \lim_{R \rightarrow \infty} \limsup_{k \rightarrow \infty} \frac{1}{\int_I G(\widehat{u}, v)\mathrm{d}v}\int_{I \times \mathbb{R}} \big(v^2+x^2\big) \mathbf{1}_{\big\{\sqrt{v^2+x^2} \geq R\big\}} \mu^k(\mathrm{d}v, \mathrm{d}x)   \nonumber \\
        & = 0 \cdot \frac{1}{\int_I G(\widehat{u}, v)\mathrm{d}v}\nonumber \\
        & =0,     \label{E:eqn:40}
    \end{align}
    where we used $\mathcal{W}_2 (\mu^k, \mu)\rightarrow0$ for the first equality. Next, for any $h \in C_b(\mathbb{R})$, considering
    \begin{align*}
        & \bigg|\int_{I \times \mathbb{R}} G(\widehat{u}^k, v)h(x)\mu^k(\mathrm{d}v, \mathrm{d}x) - \int_{I \times \mathbb{R}} G(\widehat{u}, v)h(x)\mu(\mathrm{d}v, \mathrm{d}x) \bigg| \\
        & \leq C \frac{1}{k}\sum_{i=1}^k\big|G(\widehat{u}^k, u^i) - G(\widehat{u},u^i)\big|+\bigg|\int_{I \times \mathbb{R}} G(\widehat{u}, v)h(x)\big(\mu^k(\mathrm{d}v, \mathrm{d}x) - \mu(\mathrm{d}v, \mathrm{d}x)\big) \bigg| \\
        & \rightarrow 0,
    \end{align*}
    where we used the uniform continuity of $G$ on $[0,1]^2$ for the last convergence. Combined with \eqref{E:eqn:39} and using $\int_I G(\widehat{u},v)\mathrm{d}v \neq 0$, we have for any $h \in C_b (\mathbb{R})$:
    $$
    \frac{1}{\frac{1}{k}\sum_{i=1}^k G(\widehat{u}^k,u^i)}\int_{I \times \mathbb{R}} G(\widehat{u}^k, v)h(x)\mu^k(\mathrm{d}v, \mathrm{d}x) \longrightarrow \frac{1}{\int_I G(\widehat{u},v)\mathrm{d}v} \int_{I \times \mathbb{R}} G(\widehat{u}, v)h(x)\mu(\mathrm{d}v, \mathrm{d}x), 
    $$
    which, together with \eqref{E:eqn:40}, completes the proof of $\text{II}_k \rightarrow 0$. 
    
    For the case $\int_I G(\widehat{u},v)\mathrm{d}v = 0$, we are left with
    \begin{align*}
        \lim_{k \rightarrow \infty} \text{II}_k  = \lim_{k \rightarrow \infty} 0 \cdot \bigg(\frac{1}{k}\sum_{i=1}^k G(\widehat{u}^k,u^i)\bigg) \mathcal{W}_2^2 \Big(\overline{[G \mu^k](\widehat{u}^k)}, \delta_0\Big) \leq \lim_{k \rightarrow \infty} 0 \cdot \int_{I \times \mathbb{R}} G(\widehat{u}^k, v)x^2 \mu^k(\mathrm{d}v, \mathrm{d}x)  =0,
    \end{align*}
    due to $\sup_k \int_{I \times \mathbb{R}} (u^2 + x^2) \mu^k(\mathrm{d}u, \mathrm{d}x) < \infty$.

    \noindent (b) We first notice that 
    \begin{align*}
        \big|c(t, \mu^k) - c(t, \mu)\big| & \leq \left|\int_{I \times \mathbb{R}} \big(c(t,\mu) + b(t,x, [G \mu](u))\big)^+\big(\mu^k(\mathrm{d}u, \mathrm{d}x) - \mu(\mathrm{d}u, \mathrm{d}x)\big)\right| \\
        & \quad + C \sup_{u \in I} \mathcal{WOP}_2 \big([G\mu^k](u), [G \mu](u)\big) \\
        & =: \text{IV}_k + \text{V}_k.
    \end{align*}
    From (a), we have $\lim_{k \rightarrow \infty} \text{IV}_k = 0$. For $\text{V}_k$, denote by $f^k(u) := \mathcal{WOP}_2 \big([G \mu^k](u), [G \mu](u)\big)$. The pointwise convergence of $\{f^k\}$ on $[0,1]$ comes from (a). The rest of the proof is devoted to prove the equicontinuity property. For any $a, b \in [0,1]$, by adding and subtracting terms, we have
    \begin{align*}
        \big|f^k(a) - f^k(b)\big| & = \big|f^k(a) - \mathcal{WOP}_2 \big([G \mu](a), [G \mu^k](b)\big) + \mathcal{WOP}_2 \big([G \mu](a), [G \mu^k](b)\big) - f^k(b)\big| \\
        & \leq \mathcal{WOP}_2 \big([G \mu^k](a), [G \mu^k](b)\big) +\mathcal{WOP}_2 \big([G \mu](a), [G \mu](b)\big),
    \end{align*}
    where we used the triangle inequality of $\mathcal{WOP}_2$, see Theorem 2 in \cite{LeblGoui2023}. 
    
    Suppose that $\int_I G(a,v)\mathrm{d}v \neq 0$ and $\int_I G(b,v)\mathrm{d}v \neq 0$. Then, we have from \eqref{E:eqn:28} that
    \begin{align*}
        & \mathcal{WOP}_2^2 \big([G \mu^k](a), [G \mu^k](b)\big) \\
        &  = \bigg[\frac{1}{k}\sum_{j=1}^k \big(G(a,u^j) - G(b,u^j)\big)\bigg]^2 + \bigg(\frac{1}{k}\sum_{j=1}^k G(a,u^j)\bigg)\bigg( \frac{1}{k}\sum_{j=1}^k G(b,u^j)\bigg)\mathcal{W}_2^2 \Big(\overline{[G \mu^k](a)}, \overline{[G \mu^k](b)}\Big) \\
        & \quad +  \bigg[\frac{1}{k}\sum_{j=1}^k \big(G(a,u^j) - G(b,u^j)\big)\bigg] \bigg[\int_{I \times \mathbb{R}} G(a,v) x^2 \mu^k(\mathrm{d}v, \mathrm{d}x) - \int_{I \times \mathbb{R}} G(b,v) x^2 \mu^k(\mathrm{d}v, \mathrm{d}x)\bigg] \\
        & \leq \frac{1}{k}\sum_{j=1}^k \big|G(a,u^j) - G(b,u^j)\big|^2 + C  \frac{1}{k}\sum_{j=1}^k \big|G(a,u^j) - G(b,u^j)\big| \\
        & \quad + 2 \bigg(\frac{1}{k}\sum_{j=1}^k G(a,u^j)\bigg)\bigg( \frac{1}{k}\sum_{j=1}^k G(b,u^j)\bigg) \int_\mathbb{R} x^2 \left|\overline{[G \mu^k](a)} - \overline{[G \mu^k](b)}  \right|(\mathrm{d}x) \\
        & \leq \frac{1}{k}\sum_{j=1}^k \big|G(a,u^j) - G(b,u^j)\big|^2 + C  \frac{1}{k}\sum_{j=1}^k \big|G(a,u^j) - G(b,u^j)\big| \\
        & \quad + 2 \bigg(\frac{1}{k}\sum_{j=1}^k G(a,u^j)\bigg)\bigg( \frac{1}{k}\sum_{j=1}^k G(b,u^j)\bigg) \int_\mathbb{R} x^2 \big(\mathbf{1}_{\mathrm{P}_k}(x) - \mathbf{1}_{\mathrm{N}_k}(x)\big) \left(\overline{[G \mu^k](a)} - \overline{[G \mu^k](b)}  \right)(\mathrm{d}x) \\
        & \leq \frac{1}{k}\sum_{j=1}^k \big|G(a,u^j) - G(b,u^j)\big|^2 + C  \frac{1}{k}\sum_{j=1}^k \big|G(a,u^j) - G(b,u^j)\big| \\
        & \quad + 2 \int_{I \times \mathbb{R}} \big|G(a,v)- G(b,v)\big| x^2 \mu^k(\mathrm{d}v, \mathrm{d}x) + 2 \frac{1}{k}\sum_{j=1}^k \big|G(a,u^j) - G(b,u^j)\big| \int_{I \times \mathbb{R}} x^2 \mu^k (\mathrm{d}v, \mathrm{d}x),
    \end{align*}
    where $\mathrm{P}_k$ and $\mathrm{N}_k$ are two measurable sets satisfying:
    \begin{itemize}
    \item $\mathrm{P}_k \cup \mathrm{N}_k = \mathbb{R}$ and $\mathrm{P}_k \cap \mathrm{N}_k = \emptyset$;

    \item $\mathrm{P}_k$ is a positive set and $\mathrm{N}_k$ is a negative set.
\end{itemize}
    If either $\int_I G(a,v)\mathrm{d}v=0$ or $\int_I G(b,v)\mathrm{d}v=0$ holds, then
    \begin{align*}
        \mathcal{WOP}_2^2 \big([G \mu^k](a), [G \mu^k](b)\big) \leq \frac{1}{k}\sum_{j=1}^k \big|G(a,u^j) - G(b,u^j)\big|^2 + C  \frac{1}{k}\sum_{j=1}^k \big|G(a,u^j) - G(b,u^j)\big|.
    \end{align*}
    Using the uniform continuity of $G$, i.e. for any $\epsilon > 0$, there exists some $\delta > 0$ such that if $\sqrt{|u_1 - u_2|^2 + |v_1 - v_2|^2} < \delta$, $|G(u_1, v_1) - G(u_2, v_2)| < \epsilon^2$. Consequently, if $|a-b| < \delta$, we obtain that
    \begin{align*}
        \mathcal{WOP}_2^2 \big([G \mu^k](a), [G \mu^k](b)\big) \leq \epsilon^4 + C \epsilon^2 + 4 \epsilon^2 \cdot \sup_k \int_{I \times \mathbb{R}} \big(v^2 + x^2\big) \mu^k(\mathrm{d}v, \mathrm{d}x) \leq C \epsilon^2 .
    \end{align*}
    
    Similarly, we have
    \begin{align*}
     \mathcal{WOP}_2^2 \big([G \mu](a), [G \mu](b)\big) & \leq \int_I \big|G(a,v) - G(b,v)\big|^2 \mathrm{d}v + C \int_I \big|G(a,v) - G(b,v)\big| \mathrm{d}v \\
        & \quad + 2 \int_{I \times \mathbb{R}} \big|G(a,v)-G(b,v)\big|x^2\mu(\mathrm{d}v, \mathrm{d}x)  \\
        & \quad + 2 \int_I \big|G(a,v)- G(b,v)\big| \mathrm{d}v  \int_{I \times \mathbb{R}}x^2 \mu(\mathrm{d}v, \mathrm{d}x) \\
        & \leq C \epsilon^2.
    \end{align*}
    Therefore,
    \begin{equation*}
        \big|f^k(a)- f^k(b)\big| \leq C \epsilon, \quad \text{where $C >0$ is independent of $k$}.
    \end{equation*}
    Using Lemma \ref{E:lemma:6}, the map $f^k(u)$ converges uniformly to 0 on $[0, 1]$. Thus, $\lim_{k \rightarrow \infty} \text{V}_k = 0$.
\end{proof}

\subsection{Proof of Theorem \ref{E:thm:1} (i)}
Denote by $H(x) := 1-\frac{1}{2}\mathbf{1}_{\{x \geq 0\}}$ and $\gamma(t,u,x,\mu) := b(t,x,[G \mu](u)) + c(t,\mu)$. Apart from the discontinuity of the composition of $H \circ \gamma$, the additional difficulty comes from the lack of regularity in the Wasserstein space of $B$ and $\Sigma$ in $\mu$, which prohibits the existence of a unique strong solution and hence we turn to find a weak solution. 
    
For $n \in \mathbb{N}^*$ and $t \in [0,T)$, we introduce the $n$-dyadic projection $[t]^n := 2^{-n}T\lfloor \frac{2^n t}{T}\rfloor$, where $\lfloor \cdot \rfloor$ denotes the floor function. Consider the Euler discretization of the mean-field SDE \eqref{E:eqn:6}:
\begin{equation}
    X^n_t = X_0 + \int_0^t B([s]^n, U, X^n_{[s]^n}, \mu^n_{[s]^n}) \mathrm{d}s + \int_0^t (\widehat{\sigma} H \circ \gamma) ([s]^n, U, X^n_{[s]^n}, \mu^n_{[s]^n})  \mathrm{d}W_s
\end{equation}
with $\widehat{\sigma} (t,u,x,\mu) := \sigma(t,x,[G\mu](u))$ and $\mu^n := \mathcal{L} (U, X^n)$. 

We next divide the proof into several steps. 

\vspace{1mm}

\noindent Step 1. Relative compactness of $(\mu^n)_{n \geq 1}$ in $\mathcal{P}^2_{\text{Unif}} ([0, 1] \times \mathcal{C})$. The relative compactness of $(\mathcal{L} (U, X^n))_{n \geq 1}$ in $\mathcal{P}_{\text{Unif}}^2([0,1] \times \mathcal{C})$ is equivalent to the relative compactness of $(\mathcal{L}(X^n))_{n \geq 1}$ in $\mathcal{P}^2(\mathcal{C})$ (\cite{Lacker2015}, Lemma A.2), which will hold true if and only if $(\mathcal{L}(X^n))_{n \geq 1}$ is tight and satisfies the uniform integrability condition:
\begin{equation}
    \lim_{R \rightarrow \infty} \sup_n \mathbb{E} \left[\|X^n\|^2_T \mathbf{1}_{\{\|X^n\|_T \geq R\}}\right] = 0.
    \label{E:eqn:10}
\end{equation}
since Point (ii) of Theorem 7.12 in \cite{Vill2021}. 

To show the tightness of $(\mathcal{L}(X^n))_{n \geq 1}$, we first denote by $\mathcal{T}$ the collection of all $\mathbb{F}$-stopping times. We obtain that there exists a constant $C > 0$, independent of $n$, satisfying
\begin{align}
    & \sup_{\tau \in \mathcal{T}} \mathbb{E} \left[\left|X^n_{(\tau + \epsilon) \wedge T} - X^n_\tau\right|^2\right] \nonumber \\
    & \quad = \sup_{\tau \in \mathcal{T}} \mathbb{E} \left[\left|\int_\tau^{(\tau + \epsilon) \wedge T} B([s]^n, U, X^n_{[s]^n}, \mu^n_{[s]^n}) \mathrm{d}s + \int_\tau^{(\tau + \epsilon) \wedge T} (\widehat{\sigma} H \circ \gamma) ([s]^n, U, X^n_{[s]^n}, \mu^n_{[s]^n})  \mathrm{d}W_s\right|^2\right] \nonumber \\
    & \quad \leq C \sup_{\tau \in \mathcal{T}} \mathbb{E} \left[\int_\tau^{(\tau + \epsilon) \wedge T} \left|B([s]^n, U, X^n_{[s]^n}, \mu^n_{[s]^n})\right|^2 + \left|(\widehat{\sigma} H \circ \gamma) ([s]^n, U, X^n_{[s]^n}, \mu^n_{[s]^n})\right|^2  \mathrm{d}s\right] \nonumber \\
    & \quad \leq C \epsilon, \label{E:eqn:9}
\end{align}
where we used the boundedness of $b$ and $\sigma$. Thanks to Markov's inequality, \eqref{E:eqn:9} implies for any $a > 0$ and $n \in \mathbb{N}^*$:
$$
\sup_{\tau \in \mathcal{T}} \mathbb{P} \left(\left|X^n_{(\tau + \epsilon) \wedge T} - X^n_\tau\right| \geq a\right) \leq \frac{C \epsilon}{a^2},
$$
which verifies the tightness of $(\mathcal{L}(X^n))_{n \geq 1}$ in $\mathcal{P} (\mathcal{C})$ from Aldous's criterion ((16.23) of \cite{Bill2013}). 

Using Markov's inequality, we obtain that there exists $C > 0$, independent of $n$
$$
\sup_n \mathbb{E} \left[\|X^n\|^2_T \mathbf{1}_{\{\|X^n\|_T \geq R\}}\right] \leq \sup_n \frac{1}{R^{p-2}} \mathbb{E} [\|X^n\|^p_T] \leq \frac{C}{R^{p-2}} \left(1+\mathbb{E}[|X_0|^p]\right) < \infty,
$$
since $\mu_0 \in \mathcal{P}^p_{\text{Unif}}([0,1] \times \mathbb{R})$. We conclude the verification of \eqref{E:eqn:10} by letting $R \rightarrow \infty$. Consequently, $(\mu^n)_{n \geq 1}$ is relatively compact in $\mathcal{P}^2_{\text{Unif}}([0,1] \times \mathcal{C})$, and converges to some limit $\mu^\infty \in \mathcal{P}^2_{\text{Unif}}([0,1]\times\mathcal{C})$ in $\mathcal{W}_2$, after possibly passing to a sub-sequence. For brevity of notation, we still use $(\mu^n)_{n \geq 1}$ to denote its sub-sequence. Then,  we have 
\begin{equation*}
    P^n (\mathrm{d}m, \mathrm{d}\overline{x}, \mathrm{d}t) := \delta_{\mu^n_t}(\mathrm{d}m) \mu^n_{t} (\mathrm{d}\overline{x}) \frac{\mathrm{d}t}{T} \longrightarrow P^\infty (\mathrm{d}m, \mathrm{d}\overline{x}, \mathrm{d}t) := \delta_{\mu^\infty_t} (\mathrm{d}m) \mu^\infty_{t}(\mathrm{d}\overline{x}) \frac{\mathrm{d}t}{T}, \quad \text{in $\mathcal{W}_2$},
\end{equation*}
where we used the notation $\overline{x}:=(u,x)$. Moreover, as $\mathbb{E} [|X^n_t - X^n_{[t]^n}|^2] \rightarrow 0$ as $n \rightarrow \infty$, we have
\begin{equation*}
    \lim_{n \rightarrow \infty} \mathcal{W}_2 (\overline{P}^n, P^n) =0, \quad \text{where } \overline{P}^n (\mathrm{d}m, \mathrm{d}\overline{x}, \mathrm{d}t) := \int_0^T \delta_{\mu^n_{[s]^n}} (\mathrm{d}m) \mu^n_{[s]^n} (\mathrm{d}\overline{x}) \delta_{[s]^n} (\mathrm{d}t)\frac{\mathrm{d}s}{T}.
\end{equation*}
Using the triangle inequality, we obtain that
\begin{equation}
    \lim_{n \rightarrow \infty} \mathcal{W}_2 (\overline{P}^n, P^\infty) \leq \lim_{n \rightarrow \infty} \mathcal{W}_2 (\overline{P}^n, P^n) + \lim_{n \rightarrow \infty} \mathcal{W}_2 (P^n, P^\infty) \longrightarrow 0.
    \label{E:eqn:13}
\end{equation}

\vspace{1mm}

\noindent Step 2. Denote $\mu^n_u := \mathcal{L}(X^n | U=u)$, uniquely defined up to a.e. equality. Recall that $\sigma$ is assumed to be bounded from below away from zero. It follows from Proposition A.1 of \cite{DjeteTouzi2024} that for a.e. $u$, $\mu^\infty_u$ is equivalent to the Lebesgue measure on $[0,T] \times \mathbb{R}$. Therefore, combined with the assumptions satisfied by $b$ in Assumption \ref{E:assump:1} (ii) and the convergence results proved in Lemma \ref{E:lemma:7}, the discontinuous points of the map $H \circ \gamma$ is $P^\infty$-negligible. Then, it follows from the Portmanteau theorem restricted on the space $[0,T] \times I \times \mathbb{R} \times \mathcal{P}^2_{\text{Unif}}(I \times \mathbb{R})$ that, for all bounded continuous $\varphi$:
\begin{equation}
    \lim_{n \rightarrow \infty} \int \varphi(x) \left(\widehat{\sigma} H \circ \gamma\right)^2(t, \overline{x}, m) \overline{P}^n (\mathrm{d}m, \mathrm{d}\overline{x}, \mathrm{d}t) = \int \varphi(x) \left(\widehat{\sigma} H \circ \gamma\right)^2(t, \overline{x}, m) P^\infty (\mathrm{d}m, \mathrm{d}\overline{x}, \mathrm{d}t).
    \label{E:eqn:14}
\end{equation}
due to \eqref{E:eqn:13}.

\vspace{1mm}

\noindent Step 3. For any $f \in C^2_b (\mathbb{R})$ and $t \in [0,T]$, it follows from It\^o's formula that
\begin{align*}
    & \langle f, \mu^\infty_t \rangle - \langle f, \mu_0 \rangle = \lim_{n \rightarrow \infty} \mathbb{E} \left[f(X^n_t)\right] - \mathbb{E} [f(X_0)]\\
    & = \lim_{n \rightarrow \infty} \int_0^t \mathbb{E} \left[\partial_x f (X^n_s) B([s]^n, U, X^n_{[s]^n}, \mu^n_{[s]^n}) + \frac{1}{2} \partial_x^2 f(X^n_s) \left(\widehat{\sigma}H \circ \gamma\right)^2 ([s]^n,U, X^n_{[s]^n}, \mu^n_{[s]^n})\right] \mathrm{d}s \\
    & = \lim_{n \rightarrow \infty} \int_0^t \mathbb{E} \left[\partial_x f (X^n_{[s]^n}) B([s]^n, U, X^n_{[s]^n}, \mu^n_{[s]^n}) + \frac{1}{2} \partial_x^2 f(X^n_{[s]^n}) \left(\widehat{\sigma}H \circ \gamma\right)^2 ([s]^n, U, X^n_{[s]^n}, \mu^n_{[s]^n})\right] \mathrm{d}s,
\end{align*}
where we used $\mathbb{E}[|X^n_s - X^n_{[s]^n}|^2] \leq C T/2^n$ for the last equality. It then follows from \eqref{E:eqn:14} that
$$
\langle f, \mu^\infty_t \rangle - \langle f, \mu_0 \rangle =  \int  \left[\partial_x f (x) B(s, u, x, \mu^\infty_s) + \frac{1}{2} \partial_x^2 f(x) \left(\widehat{\sigma}H \circ \gamma\right)^2 (s, u, x, \mu^\infty_s)\right]\mu^\infty_s(\mathrm{d}u, \mathrm{d}x)\mathrm{d}s,
$$
which ensures the existence of a weak solution of \eqref{E:eqn:6}. The square-integrability of the weak solution can be obtained using Fatou's lemma.

\subsection{Proof of Theorem \ref{E:thm:1} (ii)}
Let $\mu \in \mathcal{P}_\mathcal{S}(\pi^*)$ be the distribution of any weak solution of \eqref{E:eqn:6} and define the pair $(\pi^*, \mu)$. In order to prove that $(\pi^*, \mu)$ is a graphon equilibrium of the mutual holding problem, we consider the individual representative agent optimization problem $V_0 = \sup_{\beta \in \mathcal{A}} J_{\pi^*, \mu}(\beta)$, where $J_{\pi^*, \mu}(\beta) := \mathbb{E}^{\mathbb{P}^\beta_{\pi^*, \mu}}[g(X_T)]$, and the dynamics of the state process $X$ under $\mathbb{P}^\beta_{\pi^*, \mu}$ are given by
$$
\mathrm{d}X_t =\frac{\left[b(t,X_t,[G\mu_t](U))+\int \beta_t(U,X_t,v,y)B^\mu(t,v,y)\mu_t(\mathrm{d}v,\mathrm{d}y)\right]\mathrm{d}t + \sigma(t,X_t, [G\mu_t](U))\mathrm{d}W^{\pi^*, \mu, \beta}_t}{1+\mathbf{1}_{\{B^\mu(t,U,X_t)\geq 0\}}}
$$
with $B$ given in \eqref{E:eqn:3}. In order to write the corresponding HJB equation in its backward SDE form, we introduce the Hamiltonian $H$ defined for all $(t,u,x,\mu) \in [0,T] \times [0,1] \times \mathbb{R} \times \mathcal{P}^2_{\text{Unif}} ([0,1]\times \mathbb{R})$ by
\begin{align*}
    H_t (u,x,\mu,z) &:= \frac{zb(t,x,[G\mu](u)) +\sup_\beta z \int \beta_t(u,x,v,y)B^\mu(t,v,y)\mu(\mathrm{d}v, \mathrm{d}y)}{1+\mathbf{1}_{\{B^\mu(t,u,x)\geq 0\}}} \\
    & \; = \frac{zb(t,x,[G\mu](u)) + \int \left(z B^\mu(t,v,y)\right)^+\mu(\mathrm{d}v, \mathrm{d}y)}{1+\mathbf{1}_{\{B^\mu(t,u,x)\geq 0\}}},
\end{align*}
with maximizer $\hat{\beta} (t,v,y,\mu,z) := \mathbf{1}_{\{zB^\mu(t,v,y)\geq 0\}}$.

Consider the backward SDE:
$$
Y_T = g(X_T), \quad \text{and } \mathrm{d}Y_t = Z_t \mathrm{d}X_t - H_t(U,X_t, \mu_t, Z_t)\mathrm{d}t, \quad t \in [0,T], \text{ $\mu$-a.s.}
$$
Substituting the expression for the Hamiltonian $H$ and the equilibrium dynamics of $X$ corresponding to $\pi^*$, we rewrite the dynamics of $Y$ in terms of the Brownian motion $W$ as:
\begin{equation}
    Y_T = g(X_T), \quad \text{and } \mathrm{d}Y_t = F_t(U,X_t,Z_t)\mathrm{d}t + Z_t \Sigma(t,U,X_t, \mu_t)\mathrm{d}W_t, \text{ $\mathbb{P}$-a.s.,}
    \label{E:eqn:7}
\end{equation}
where the generator $F$ is given by
$$
F_t(u,x,z) := \frac{\int\left[z(B^\mu)^+ - (zB^\mu)^+\right](t,v,y) \mu_t(\mathrm{d}v, \mathrm{d}y)}{1+\mathbf{1}_{\{B^\mu(t,u,x)\geq 0\}}}.
$$
It follows from the standard result of \cite{PardPeng1990} that this backward SDE has a unique solution $(Y,Z)$ satisfying $\mathbb{E} [\sup_{t \leq T}|Y_t|^2 + \int_0^T |Z_t|^2 \mathrm{d}t] < \infty$.

We next prove that $Z \geq 0$, $\mathbb{P}$-a.s. Similar to the proof of Lemma 6.1 in \cite{DjeteTouzi2024}, we consider the candidate solution of the backward SDE:
$$
\overline{Y}_T = g(X_T), \quad \text{and } \mathrm{d}\overline{Y}_t = \overline{Z}_t \Sigma (t,U,X_t, \mu_t)\mathrm{d}W_t, \quad \text{$\mathbb{P}$-a.s.}
$$
The $\overline{Y}$-component of the unique solution of this equation is the $\mathbb{H}^2(\mu)$-process $\overline{Y}_t = \mathbb{E}^\mathbb{P} [g(X_T)|\mathcal{F}_t]$, $t \leq T$, and we are further reduced to verify that $\overline{Z} \geq 0$ by the unique solvability of \eqref{E:eqn:7}. Let us assume that $B$ admits bounded derivative w.r.t. $u$ and $\Sigma$ admits bounded derivative w.r.t. $(u,x)$. By the Clark-Ocone formula, we have
$$
\overline{Z}_t \Sigma(t, U, X_t, \mu_t) = \mathbb{E}^\mathbb{P} \left[(0, g^\prime(X_T)) D_t (U,X_T) | \mathcal{F}_t\right],
$$
where $D_t (U, X_T)$ is the Malliavin derivative of the diffusion process $(U,X)$, whose dynamics are given by
$$
\mathrm{d}\binom{U}{X_t} = \binom{0}{B^\mu (t, U, X_t)} \mathrm{d}t + \binom{0}{\Sigma^\mu(t,U,X_t)} \mathrm{d}W_t,
$$
and the associated first variation process is defined by
$$
\mathrm{d}Y_t = \left(
\begin{array}{cc}
0 & 0 \\
\partial_u B^\mu_t & \partial_x B^\mu_t
\end{array}
\right) Y_t \mathrm{d}t + \left(
\begin{array}{cc}
0 & 0 \\
\partial_u \Sigma^\mu_t & \partial_x \Sigma^\mu_t
\end{array}
\right) Y_t \mathrm{d}W_t, \quad Y_0 = I_2
$$
with $I_2$ the identity matrix of $\mathbb{R}^2$. More precisely, $Y$ solves the system
$$
\left\{
\begin{aligned}
    &Y^{1,1} = 1, \text{ and } Y^{1,2} = 0 \\
    &\mathrm{d}Y^{2, 1}_t = \left(\partial_u B_t^\mu + \partial_x B^\mu_t Y^{2, 1}_t\right) \mathrm{d}t + \left(\partial_u \Sigma_t^\mu + \partial_x \Sigma^\mu_t Y^{2, 1}_t\right) \mathrm{d}W_t, \quad Y^{2, 1}_0 = 0 \\
    & \mathrm{d}Y^{2, 2}_t = \partial_x B^\mu_t Y^{2, 2}_t \mathrm{d}t + \partial_x \Sigma^\mu_t Y^{2, 2}_t \mathrm{d}W_t, \quad Y^{2, 2}_0 = 1
\end{aligned}.
\right.
$$
The malliavin derivative of $(U, X)$ is then given by
\begin{align*}
D_t (U, X_T) = Y_T Y^{-1}_t \binom{0}{\Sigma^\mu_t} &= 
\left(
\begin{array}{cc}
1 & 0 \\
Y^{2, 1}_T & Y^{2, 2}_T
\end{array}
\right)
\left(
\begin{array}{cc}
1 & 0 \\
-Y^{2, 1}_t / Y^{2, 2}_t & 1/Y^{2, 2}_t
\end{array}
\right)\binom{0}{\Sigma^\mu_t} \\
& = \binom{0}{\Sigma^\mu_t Y^{2, 2}_T / Y^{2, 2}_t}, \quad t \leq T.
\end{align*}
Hence,
$
\overline{Z}_t = \mathbb{E}^\mathbb{P} [g^\prime(X_T) Y^{2, 2}_T/Y^{2, 2}_t | \mathcal{F}_t] \geq 0
$ since the tangent process $Y^{2,2}$ is positive and $g$ is nondecreasing. The rest of the proof follows in the same way as the one of Lemma 6.1 in \cite{DjeteTouzi2024}. Notice that under $Z \geq 0$, the generator $F$ of \eqref{E:eqn:7} vanishes, and we obtain that $Y_t = \mathbb{E}^\mathbb{P}[g(X_T) | \mathcal{F}_t]$, $t \in [0,T]$.

To complete the proof, we now show that $\mathbb{E}^\mathbb{P}[g(X_T)] = J_{\pi^*, \mu} (\pi^*) \geq J_{\pi^*, \mu}(\beta)$, for all $\beta \in \mathcal{A}$. Indeed, using the representation of $g(X_T)$ in terms of $Y$:
$$
J_{\pi^*, \mu}(\beta) = \mathbb{E}^{\mathbb{P}^\beta_{\pi^*, \mu}}[g(X_T)] = Y_0 + \mathbb{E}^{\mathbb{P}^\beta_{\pi^*, \mu}}\left[\int_0^T Z_t \mathrm{d}X_t - H_t(U,X_t, \mu_t,Z_t)\mathrm{d}t\right].
$$
Substituting the dynamics of $X$, it follows from appropriate localization that 
\begin{align*}
    & J_{\pi^*, \mu}(\beta) \\
    & \quad = Y_0 + \mathbb{E}^{\mathbb{P}^\beta_{\pi^*, \mu}} \left[\int_0^T \left(Z_t \frac{b(t,X_t,[G\mu_t](U)) + \int \beta_t B^\mu(t,v,y)\mu_t(\mathrm{d}v, \mathrm{d}y)}{1+\mathbf{1}_{\{B^\mu(t,U,X_t)\geq 0\}}} - H_t(U,X_t,\mu_t,Z_t)\right)\mathrm{d}t\right] \\
    & \quad \leq Y_0 = \mathbb{E}^\mathbb{P}[g(X_T)],
\end{align*}
where the last inequality follows from the definition of $H$. By the same argument, we see that the control $\pi^* \in \mathcal{A}$ allows to reach the last upper bound, i.e. $J_{\pi^*, \mu}(\pi^*) = \mathbb{E}^\mathbb{P}[g(X_T)]$.

\section{Proof of Theorem \ref{E:thm:2} (i)} \label{E:sec:1}
Recall that $\Pi := \Pi^N$ is defined in \eqref{E:eqn:18}. In the next lemma, for any $\mathbf{x}:=(x^1, \cdots, x^N) \in \mathbb{R}^N$ and $\mathbf{u}:=(u^1, \cdots, u^N) \in I^N_1 \times \cdots \times I^N_N$. Denote $m^{i, N}(\mathbf{x}) := \frac{1}{N} \sum_{j=1}^N \xi^N_{ij} \delta_{x^j}$.
\begin{lemma}
    There exists $C > 0$, independent of $N$ and $\beta \in [0,1]^N$, such that for all $t \in [0, T]$, $\mathbf{x}, \mathbf{y} \in \mathbb{R}^N$, we have
    \begin{align*}
        & \big|\mathbf{\Sigma}^{k,k} (t, \Pi^{-i}_t (\beta),\mathbf{u},\mathbf{x}) - \mathbf{\Sigma}^{k,k} (t, \Pi_t ,\mathbf{u}, \mathbf{x})\big| + \big|\mathbf{B}^k (t, \Pi^{-i}_t (\beta),\mathbf{u},\mathbf{x}) - \mathbf{B}^k (t, \Pi_t ,\mathbf{u},\mathbf{x})\big| \leq \frac{C}{N}, \; \text{for $k \neq i$}, \\
        & \big|\mathbf{\Sigma}^{k,k} (t, \Pi_t ,\mathbf{u},\mathbf{x}) - \mathbf{\Sigma}^{k,k} (t, \Pi_t ,\mathbf{u},\mathbf{y})\big| + \big|\mathbf{B}^k (t, \Pi_t ,\mathbf{u},\mathbf{x}) - \mathbf{B}^k (t, \Pi_t ,\mathbf{u},\mathbf{y})\big| \leq C \bigg(\phi^k + \frac{1}{N} \sum_{j=1}^N \phi^j\bigg) (t, \mathbf{x}, \mathbf{y}), \\
        & \text{where $\phi^j (t, \mathbf{x}, \mathbf{y}) := \big|(b, \sigma)(t,x^j, m^{j,N}(\mathbf{x})) - (b,\sigma)(t,y^j,m^{j,N}(\mathbf{y}))\big|$, and} \\
        & \sup_{1 \leq q \neq e \leq N} \big|\mathbf{\Sigma}^{e,q}(t, \Pi_t,\mathbf{u}, \mathbf{x})\big| + \sup_{1 \leq k \leq N} \left|\mathbf{\Sigma}^{k,k}(t, \Pi_t, \mathbf{u},\mathbf{x}) - \frac{\sigma(t, x^k, m^{k,N}(\mathbf{x}))}{1+\pi_t^k}\right| \leq \frac{C}{N}, \\
        & \big|\mathbf{\Sigma}^{k,k}(t, \Pi_t, \mathbf{u},\mathbf{x})\big| + \big|\mathbf{B}^k(t, \Pi_t, \mathbf{u},\mathbf{x})\big| \leq C\bigg[\big|(b,\sigma)(t,x^k, m^{k,N}(\mathbf{x}))\big| + \frac{1}{N}\sum_{j=1}^N \big|(b,\sigma)(t,x^j, m^{j,N}(\mathbf{x}))\big|\bigg].
    \end{align*}
    \label{E:lemma:2}
\end{lemma}

\begin{proof}
    Similar to Lemma 7.1 in \cite{DjeteTouzi2024}, hence omitted. 
\end{proof}

For all $N \in \mathbb{N}^*$, consider $(\beta^i)_{1 \leq i \leq N} := (\beta^{i,1}, \cdots, \beta^{i, N})_{1 \leq i \leq N}$. Recall that the $\mathbb{P}$-dynamics of the corresponding $i$-deviated equity process $\mathbf{X}^i := (X^{i,1}, \cdots, X^{i,N})$ satisfies for all $k =1, \cdots, N$:
\begin{align*}
    \mathrm{d}X^{i,k}_t & = \mathbf{B}^k (t, \Pi_t^{-i}(\beta^i_t), \mathbf{u}, \mathbf{X}^i_t) \mathrm{d}t + \sum_{q \neq i} \mathbf{\Sigma}^{k,q}(t, \Pi^{-i}_t(\beta^i_t), \mathbf{u}, \mathbf{X}^i_t) \mathrm{d}W^q_t \\
    & \quad + \mathbf{\Sigma}^{k,i} (t, \Pi^{-i}_t (\beta^i_t), \mathbf{u},\mathbf{X}^i_t) (\mathrm{d}W^i_t - \psi^i_t \mathrm{d}t), \text{ with $X^{i,k}_0 = X^k_0 \sim \mu_{u^k,0}$}.
\end{align*}

\begin{lemma}
    For all $k=1, \cdots, N$, we obtain that
    \begin{equation}
        \mathbb{E}^\mathbb{P} \bigg[\sup_{0 \leq t \leq T} \big|X_t^{i,k} - X_t^k\big|^2\bigg] \leq \frac{C}{N},
        \label{E:eqn:20}
    \end{equation}
    with $C$ independent of $N$. Denote by $\mu^N := \frac{1}{N} \sum_{j=1}^N \delta_{(u^j, X^j)}$ and $\mu^{i,N} := \frac{1}{N} \sum_{j=1}^N \delta_{(u^j, X^{i,j})}$ the corresponding empirical measures. Consequently, we have
    $$
    \lim_{N \rightarrow \infty} \frac{1}{N} \sum_{i=1}^N \mathbb{E}^\mathbb{P} \bigg[\sup_{0 \leq t \leq T} \big|X^i_t - X^{i,i}_t\big|^2\bigg] = 0 \text{ and } \lim_{N \rightarrow \infty} \frac{1}{N} \sum_{i=1}^N \mathbb{E}^\mathbb{P} \left[\mathcal{W}_2^2 (\mu^{i,N}, \mu^N)\right] = 0.
    $$
    \label{E:lemma:3}
\end{lemma} 

\begin{proof}
    For simplicity, we denote $B^{i,k}_t (\mathbf{x}) := \mathbf{B}^k(t, \Pi_t^{-i}(\beta^i_t), \mathbf{u}, \mathbf{x})$, $B^k_t(\mathbf{x}) := \mathbf{B}^k(t, \Pi_t, \mathbf{u},\mathbf{x})$, and similarly, we denote $\Sigma^{i,k,q}_t (\mathbf{x}) := \mathbf{\Sigma}^{k,q}(t, \Pi^{-i}_t(\beta^i_t), \mathbf{u}, \mathbf{x})$, $\Sigma^{k,q}_t(\mathbf{x}) := \mathbf{\Sigma}^{k,q}(t, \Pi_t, \mathbf{u}, \mathbf{x})$. Considering for any $k=1, \cdots, N$ and assuming that $\frac{1}{N} \sum_{j=1}^N \xi^N_{kj} \neq 0$, we obtain that
    \begin{align}
        \mathcal{WOP}_2^2 \big(m^{k,N}(\mathbf{X}^i_t), m^{k,N}(\mathbf{X}_t)\big) & = \bigg(\frac{1}{N} \sum_{j=1}^N \xi^N_{k j}\bigg)^2 \mathcal{W}_2^2 \bigg(\frac{\frac{1}{N} \sum_{j=1}^N \xi^N_{kj}\delta_{X^{i,j}_t}}{\frac{1}{N}\sum_{j=1}^N \xi_{kj}^N}, \frac{\frac{1}{N} \sum_{j=1}^N \xi^N_{kj}\delta_{X^{j}_t}}{\frac{1}{N}\sum_{j=1}^N \xi_{kj}^N}\bigg) \nonumber \\
        & \leq \frac{1}{N} \sum_{j=1}^N \xi^N_{k j} \cdot \frac{1}{N} \sum_{l=1}^N \xi^{N}_{kl} \mathcal{W}_2^2\left(\delta_{X^{i,l}_t}, \delta_{X^l_t}\right) \nonumber \\
        & \leq \frac{1}{N} \sum_{l=1}^N \big|X^{i,l}_t - X^l_t\big|^2,
        \label{E:eqn:19}
    \end{align}
    where we used the convexity of $\mathcal{W}_2^2 (\cdot, \cdot)$ for the second line and the boundedness of graphon for the last line. It is straightforward that \eqref{E:eqn:19} also holds for $\frac{1}{N} \sum_{j=1}^N \xi^N_{kj} = 0$. By the properties proved in Lemma \ref{E:lemma:2} and the inequality \eqref{E:eqn:19} obtained above, together with Assumption \ref{E:assump:1}, we show the existence of a constant $C$, independent of $N$, such that for any $k \in \{1, \cdots, N\}$:
    \begin{align*}
        \sup_{k} \mathbb{E}^\mathbb{P} \bigg[\sup_{t \in [0,T]} \big|X^{i,k}_t - X^k_t\big|^2\bigg]& \leq C \bigg(\frac{1}{N} + \sup_{k} \mathbb{E}^\mathbb{P} \bigg[\int_0^T \bigg(\big|X^{i,k}_s -X^k_s\big|^2 + \frac{1}{N} \sum_{j=1}^N \big|X^{i,j}_s - X^j_s\big|^2\bigg) \mathrm{d}s\bigg]\bigg) \\
        & \leq C \bigg(\frac{1}{N} + \int_0^T \sup_k \mathbb{E}^\mathbb{P}\bigg[\sup_{r \in [0,s]} \big|X_r^{i,k} - X^k_r\big|^2\bigg] \mathrm{d}s\bigg),
    \end{align*}
    which completes the proof of \eqref{E:eqn:20} by the Gronwall's inequality.
\end{proof}

The following lemma is essential for the proof under the context of general kernel.
\begin{lemma}
    There exists some positive constant $C$, independent of N, such that for any $t \in [0,T]$,
    $$
    \sup_{1 \leq i \leq N} \mathcal{WOP}_2^2 \Big([G^N \mu^N_t](u^i), [G \mu^N_t](u^i)\Big) \leq C \left(N^2\big\|G^N -G \big\|_\square +N^2 \sup_{1 \leq i,j \leq N} \mathcal{R}^N_{ij}\right)\bigg(1+ \frac{1}{N}\sum_{k=1}^N \big|X^k_t\big|^2\bigg),
    $$
    with $\mu^N_t := \frac{1}{N}\sum_{j=1}^N \delta_{(u^j, X^j_t)}$, where $\mathcal{R}^{N}_{ij}$ is defined as
    \begin{equation*}
        \mathcal{R}^N_{ij} := \int_{I^N_i \times I^N_j} \big|G(u,v) - G(u^i, u^j)\big| \mathrm{d}u \mathrm{d}v.
    \end{equation*}
    \label{E:lemma:4}
\end{lemma}
\begin{proof}
    For brevity of notation, we denote $G_{i j}:=G(u^i,u^j)$ and set $x_0 :=0$. By the property of the $\mathcal{WOP}_2$ metric, we have for the case of $\frac{1}{N}\sum_{j=1}^N G_{ij} \neq 0$ and $\frac{1}{N}\sum_{j=1}^N \xi^N_{ij} \neq 0$:
    \begin{align*}
        & \mathcal{WOP}_2^2 \Big([G^N \mu^N_t](u^i), [G \mu^N_t](u^i)\Big) \\
        & = \bigg[\frac{1}{N} \sum_{j=1}^N \left(\xi^N_{ij}-G_{ij}\right)\bigg]^2 + \bigg[\frac{1}{N} \sum_{j=1}^N \left(\xi^N_{ij}-G_{ij}\right)\bigg] \bigg[\frac{1}{N} \sum_{j=1}^N \left(\xi^N_{ij}-G_{ij}\right) \big|X_t^j\big|^2\bigg] \\
        & \quad + \bigg(\frac{1}{N} \sum_{j=1}^N \xi^N_{ij}\bigg) \bigg(\frac{1}{N}\sum_{j=1}^N G_{ij}\bigg) \mathcal{W}_2^2 \Big(\overline{[G^N \mu^N_t](u^i)}, \overline{[G \mu^N_t](u^i)}\Big)\\
        & =: \text{VI}_{i,N} +\text{VII}_{i,N}.
    \end{align*}
    
    By adding and subtracting terms, we have
    \begin{align}
        \left|\frac{1}{N}\sum_{j=1}^N \big(\xi^N_{ij} - G_{ij}\big)\right| & \leq N\left|\int_{I^N_i \times I} \big(G^N(u,v) - G(u,v)\big) \mathrm{d}u\mathrm{d}v\right| + N \sum_{j=1}^N \int_{I^N_i \times I^N_j} \big|G(u,v)-G_{ij}\big| \mathrm{d}u \mathrm{d}v \nonumber \\
        & \leq N \big\|G-G^N\big\|_\square + N \sum_{j=1}^N \mathcal{R}^N_{ij} \nonumber \\
        & \leq N^2 \big\|G-G^N\big\|_\square + N^2 \sup_{1 \leq j \leq N} \mathcal{R}^N_{ij} \label{E:eqn:55},
    \end{align}
    where we used the definition of the cut-norm for the second inequality, and similarly,
    \begin{align}
        \big|\xi^N_{ij} -G_{ij}\big| & \leq N^2 \left|\int_{I^N_i \times I^N_j} \big(G^N(u,v) -G(u,v)\big) \mathrm{d}u\mathrm{d}v\right| + N^2 \int_{I^N_i \times I^N_j} \big|G(u,v) -G_{ij}\big| \mathrm{d}u\mathrm{d}v \nonumber \\
        & \leq N^2 \big\|G-G^N\big\|_\square + N^2 \mathcal{R}^N_{ij}. \label{E:eqn:56}
    \end{align}
    Therefore, $\text{VI}_{i,N}$ can be upper bounded by 
    \begin{equation*}
        \text{VI}_{i,N} \leq \left(N^2 \big\|G-G^N\big\|_\square + N^2 \sup_{1 \leq j \leq N} \mathcal{R}^N_{ij}\right)^2\left(1 + \frac{1}{N} \sum_{k=1}^N |X^k_t|^2\right).
    \end{equation*}
    
    To bound $\text{VII}_{i,N}$, we control the 2-Wasserstein distance by means of a weighted total variation distance, recall \eqref{E:eqn:28}, and then use \eqref{E:eqn:55}-\eqref{E:eqn:56}
    \begin{align*}
        \text{VII}_{i,N} & \leq  2\bigg(\frac{1}{N} \sum_{j=1}^N \xi^N_{ij}\bigg) \bigg(\frac{1}{N}\sum_{j=1}^N G_{ij}\bigg) \frac{1}{N} \sum_{k=1}^N \bigg|\frac{\xi^N_{ik}}{\frac{1}{N} \sum_{j=1}^N \xi^N_{ij}} -\frac{G_{ik}}{\frac{1}{N} \sum_{j=1}^N G_{ij}}\bigg| \big|X^k_t\big|^2 \\
        & \leq 2 \bigg(\frac{1}{N}\sum_{j=1}^N G_{ij}\bigg) \frac{1}{N} \sum_{k=1}^N \big|\xi^N_{ik} -G_{ik}\big| \big|X^k_t\big|^2 + 2 \frac{1}{N} \sum_{k=1}^N G_{ik} \bigg|\frac{1}{N} \sum_{j=1}^N G_{ij}-\frac{1}{N} \sum_{j=1}^N \xi^N_{ij}\bigg| \big|X^k_t\big|^2 \\
        & \leq  4 \left(N^2 \big\|G-G^N\big\|_\square + N^2 \sup_{1 \leq j \leq N} \mathcal{R}^N_{ij}\right) \frac{1}{N}\sum_{k=1}^N |X^k_t|^2.
    \end{align*}
    
    Next we deal with degenerate situations, i.e. either $\frac{1}{N}\sum_{j=1}^N \xi^N_{ij} = 0$ or $\frac{1}{N}\sum_{j=1}^N G_{ij} = 0$. For the case of $\frac{1}{N}\sum_{j=1}^N \xi^N_{ij} = 0$, using the almost sure finiteness of $X^j_t$, we obtain that
    \begin{align*}
        \mathcal{WOP}_2^2 \Big([G^N \mu^N_t](u^i), [G \mu^N_t](u^i)\Big) \leq \left(N^2 \big\|G-G^N\big\|_\square + N^2 \sup_{1 \leq j \leq N} \mathcal{R}^N_{ij}\right)^2\left(1 + \frac{1}{N} \sum_{k=1}^N |X^k_t|^2\right).
    \end{align*}
\end{proof}

\subsection{Proof of Theorem \ref{E:thm:2} (i): sampling kernel} \label{E:subsec:1}
Using Lemma \ref{E:lemma:3}, one can verify that
$$
\lim_{N \rightarrow \infty} \mathcal{W}_2 \bigg(\frac{1}{N}\sum_{i=1}^N \mathcal{L}^\mathbb{P} \left(Z^i, u^i, X^{i,i}, \mu^{i,N}\right), \frac{1}{N}\sum_{i=1}^N \mathcal{L}^\mathbb{P}\left(Z^i,u^i,X^i,\mu^N\right)\bigg) = 0.
$$
Consequently, we may instead focus on the sequence $\big(\frac{1}{N}\sum_{i=1}^N \mathcal{L}^\mathbb{P}(Z^i,u^i, X^i,\mu^N)\big)_{N \geq 1}$. Recall the dynamics of $(Z^i, X^i)$:
\begin{align}
    \mathrm{d}X^i_t &= B^i(t, \mathbf{u}, \mathbf{X}_t, \mu^N_t) \mathrm{d}t + \sum_{j=1}^N \Sigma^{i,j}(t, \mathbf{u}, \mathbf{X}_t, \mu^N_t) \mathrm{d}W^j_t, \label{E:eqn:30} \\
    \frac{\mathrm{d}Z^i_t}{Z^i_t} & = \frac{\mathbf{B}^i (t, \Pi^{-i}_t(\beta^i_t), \mathbf{u},\mathbf{X}_t) - \mathbf{B}^i(t, \Pi_t, \mathbf{u},\mathbf{X}_t)}{\mathbf{\Sigma}^{i,i}(t, \Pi^{-i}_t (\beta^i_t), \mathbf{u},\mathbf{X}_t)} \mathrm{d}W^i_t. \label{E:eqn:31}
\end{align} 
Firstly, we notice that $B^i_t := \mathbf{B}^i (t, \Pi_t, \mathbf{u},\mathbf{X}_t) = B^i(t, \mathbf{u}, \mathbf{X}_t, \mu^N_t)$ and $\Sigma^{i,j}_t := \mathbf{\Sigma}^{i,j} (t, \Pi_t, \mathbf{u},\mathbf{X}_t) = \Sigma^{i,j}(t, \mathbf{u}, \mathbf{X}_t, \mu^N_t)$. Denote $B^{i,i}_t (\mathbf{X}_t) := \mathbf{B}^i (t, \Pi^{-i}_t(\beta^i_t), \mathbf{u}, \mathbf{X}_t)$ and $\Sigma^{i,i,i}_t (\mathbf{X}_t) :=\mathbf{\Sigma}^{i,i}(t, \Pi^{-i}_t (\beta^i_t), \mathbf{u},\mathbf{X}_t)$. 

We define 
\begin{equation}
    q^{i,N}_t (\mathrm{d}\alpha, \mathrm{d}u, \mathrm{d}x) := \frac{1}{N} \sum_{j = 1}^N \delta_{(\beta^{i,j}_t, u^j, X^j_t)} (\mathrm{d}\alpha, \mathrm{d}u, \mathrm{d}x) \in \mathcal{P} ([0,1]^2 \times \mathbb{R}),
\label{E:eqn:27}
\end{equation}
and for each $\mu \in \mathcal{P}([0,1] \times \mathbb{R})$, one considers the Borel set $\mathbb{Z}_\mu$ defined as follows:
$$
\mathbb{Z}_\mu := \left\{m \in \mathcal{P}([0,1]^2 \times \mathbb{R}): m([0,1], \mathrm{d}u, \mathrm{d}x) = \mu(\mathrm{d}u, \mathrm{d}x)\right\}.
$$
Using $G^N(u^i,u^j) = G(u^i,u^j)$ for all $i,j$, we obtain that
\begin{equation}
    [G^N \mu^N_t](u^i) := \frac{1}{N}\sum_{j=1}^N G^N(u^i, u^j) \delta_{X^j_t} = \frac{1}{N}\sum_{j=1}^N G(u^i, u^j) \delta_{X^j_t} =: [G \mu^N_t](u^i).
    \label{E:eqn:44}
\end{equation}

Using \eqref{E:eqn:23}, \eqref{E:eqn:44} and by adding and subtracting terms, we have
\begin{align*}
    & (1+\pi^i_t) B^i_t = \frac{\frac{1}{N} \sum_{j=1}^N \frac{\pi_t^j}{1+\pi_t^j} b(t, X^j_t, [G^N \mu^N_t](u^j))}{1-\frac{1}{N}\sum_{k=1}^N \frac{\pi^k_t}{1+\pi_t^k}} +b (t,X^i_t, [G^N \mu^N_t](u^i)) \\
    & = \frac{\frac{1}{N}\sum_{j=1}^N \frac{\pi^j_t}{1+\pi^j_t} \left(b(t, X^j_t, [G \mu^N_t](u^j)) + c(t,\mu^N_t)\right)- c(t, \mu^N_t) \frac{1}{N}\sum_{j=1}^N \frac{\pi^j_t}{1+\pi^j_t}}{1-\frac{1}{N}\sum_{k=1}^N\frac{\pi^k_t}{1+\pi^k_t}} + b (t,X^i_t, [G \mu^N_t](u^i)) \\
    & = \frac{\frac{1}{2 N}\sum_{j=1}^N \left(b(t, X^j_t, [G \mu^N_t](u^j)) + c(t,\mu^N_t)\right)^+ - c(t, \mu^N_t) \frac{1}{N}\sum_{j=1}^N \frac{\pi^j_t}{1+\pi^j_t}}{1-\frac{1}{N}\sum_{k=1}^N\frac{\pi^k_t}{1+\pi^k_t}} + b (t,X^i_t, [G \mu^N_t](u^i)) \\
    & = c(t, \mu^N_t) + b (t,X^i_t, [G \mu^N_t](u^i)),
\end{align*}
which further gives that
\begin{equation}
    B^i_t = \frac{c(t, \mu^N_t) + b (t,X^i_t, [G \mu^N_t](u^i))}{1+\pi^i_t} = B(t,u^i, X^i_t, \mu^N_t),
    \label{E:eqn:45}
\end{equation}
with $B$ given in \eqref{E:eqn:3}. Using \eqref{E:eqn:23}, \eqref{E:eqn:27} and \eqref{E:eqn:45} gives
\begin{align}
    & \bigg(1 + \frac{N-1}{N} \pi^i_t + \frac{1}{N} \beta^{i,i}_t\bigg) B^{i,i}_t = \frac{1}{N} \sum_{j =1}^N \beta_t^{i,j} B_t^{i,j} + b(t,X_t^i, [G^N \mu^N_t](u^i)) \nonumber \\
    & \quad = \frac{1}{N} \sum_{j=1}^N \beta^{i,j}_t B(t, u^j, X^j_t, \mu^N_t) + b(t,X_t^i, [G \mu^N_t](u^i)) + \frac{1}{N}\sum_{j=1}^N \beta^{i,j}_t (B^{i,j}_t - B^j_t)  \nonumber \\
    & \quad = \int \alpha B(t,u,x,\mu^N_t) q^{i,N}_t (\mathrm{d}\alpha, \mathrm{d}u, \mathrm{d}x) + b(t,X_t^i, [G \mu^N_t](u^i))  + \frac{1}{N}\sum_{j=1}^N \beta^{i,j}_t (B^{i,j}_t - B^j_t). 
    \label{E:eqn:29}
\end{align} 
Similarly, we have
\begin{equation}
    \Sigma^{i,j}_t = \frac{\frac{1}{N}A^j_t(\mathbf{u},\mathbf{X}_t)\sigma(t,X^j_t, [G^N \mu^N_t](u^j))}{1+\pi^i_t} = \frac{\frac{1}{N}A^j_t(\mathbf{u},\mathbf{X}_t)\sigma(t,X^j_t, [G \mu^N_t](u^j))}{1+\pi^i_t} , \quad \text{for } j \neq i,
\end{equation}
and for the case of $j=i$, we have
\begin{equation}
    \Sigma^{i,i}_t = \frac{\left(\frac{1}{N}A^i_t(\mathbf{u},\mathbf{X}_t)+1\right) \sigma(t,X^i_t, [G^N \mu^N_t](u^i)) }{1+\pi^i_t}  = \Sigma (t,u^i, X^i_t, \mu^N_t) +\frac{A^i_t(\mathbf{u},\mathbf{X}_t) \sigma(t,X^i_t, [G \mu^N_t](u^i)) }{N(1+\pi^i_t)},
\end{equation}
where $\Sigma$ is defined in \eqref{E:eqn:4}. For $\Sigma^{i,i,i}_t$, it follows from \eqref{E:eqn:24} that
\begin{align}
    \bigg(1 + \frac{N-1}{N}\pi^i_t + \frac{1}{N} \beta^{i,i}_t\bigg) \Sigma^{i,i,i}_t = \frac{1}{N}\sum_{j=1}^N \beta^{i,j}_t \Sigma^{i,j,i}_t + \sigma(t,X^i_t, [G \mu^N_t](u^i)).
    \label{E:eqn:26}
\end{align}

On some filtered probability space $(\Omega, \mathcal{F}, \mathbb{F}, \mathbb{P})$, suppose that $(U, X, Z, W)$ is a weak solution of the system:
\begin{align*}
    & \mathrm{d}X_t = B(t, U, X_t, \mu_t) \mathrm{d}t + \Sigma(t,U,X_t, \mu_t) \mathrm{d}W_t \quad \text{with } \mu_t = \mathcal{L}(U,X_t) \in \mathcal{P}_{\text{Unif}}([0,1]\times \mathbb{R}), \\
    & \frac{\mathrm{d}Z_t}{Z_t} = \int \frac{1+\pi(t, U, X_t, \mu_t)}{\sigma (t, X_t, [G \mu_t](U))} \bigg[\frac{\int \alpha B(t,u,x,\mu_t) q(\mathrm{d}\alpha, \mathrm{d}u, \mathrm{d}x) + b(t,X_t,[G\mu_t](U))}{1+\pi(t,U, X_t, \mu_t)} \\
    & \hspace{5cm} - B(t,U,X_t, \mu_t)\bigg] M(\mathrm{d}q, \mathrm{d}t),
\end{align*}
where $M(\mathrm{d}q, \mathrm{d}t)$ is a martingale measure with quadratic variation $\Lambda:=(\Lambda_t)_{t \in [0,T]}$, which is a set of $\mathcal{P}(\mathcal{P}([0,1]^2 \times \mathbb{R}))$-valued $\mathbb{F}$-predictable processes, satisfying $\Lambda_t(\mathbb{Z}_{\mu_t}) =1$, $\mathrm{d}t \otimes \mathrm{d}\mathbb{P}-a.e.$, and $W_t = M(\mathcal{P}([0,1]^2 \times \mathbb{R}) \times [0,t])$, $t \in [0,T]$, defines a Brownian motion (see, e.g. \cite{Karoui1990}). Using the martingale problem formulation, $(U,X,Z,W)$ is a weak solution of the last SDE if for any $f \in C_b^2 ([0,1] \times \mathbb{R}^3)$, the process
\begin{align*}
    M_t^{f}(U, X, Z, W, \Lambda, \mu):= f(U, X_t, Z_t, W_t) - \int_0^t \int \mathcal{L}^{\mu_s,q}_s f(U,X_s,Z_s,W_s) \Lambda_s(\mathrm{d}q)\mathrm{d}s, \; t \in [0,T],
\end{align*}
is a $(\mathbb{F}, \mathbb{P})$-martingale, where $\mathcal{L}^{\mu_s,q}_s$ is the generator of $(U, X, Z, W)$ defined by
\begin{align*}
    & \mathcal{L}^{\mu_s,q}_s f(U,X_s,Z_s,W_s) := \partial_x f(U, X_s, Z_s, W_s) B^\mu(s,U,X_s) + \frac{1}{2}\partial_x^2 f(U, X_s, Z_s, W_s) \left(\Sigma^\mu(s,U,X_s)\right)^2 \\
    & \quad + \frac{1}{2} \partial_w^2 f(U, X_s, Z_s, W_s) + \partial_{x w} f(U, X_s, Z_s, W_s)\Sigma^\mu (s, U, X_s) \\
    & \quad + \frac{1}{2} \partial_z^2 f(U, X_s, Z_s, W_s) Z_s^2 \Bigg\{\bigg[\frac{\int \alpha B^\mu (s,u,x) q(\mathrm{d}\alpha, \mathrm{d}u, \mathrm{d}x) + b(s,X_s,[G\mu_s](U))}{1+\pi(s,U, X_s, \mu_s)} \\
    & \quad \hspace{5cm} - B^\mu (s,U,X_s)\bigg]\frac{1+\pi(s, U, X_s, \mu_s)}{\sigma (s, X_s, [G \mu_s](U))}\Bigg\}^2 \\
    & \quad + \partial_{x z} f(U, X_s, Z_s, W_s) Z_s \bigg[\frac{\int \alpha B^\mu (s,u,x) q(\mathrm{d}\alpha, \mathrm{d}u, \mathrm{d}x) + b(s,X_s,[G\mu_s](U))}{1+\pi(s,U, X_s, \mu_s)} - B^\mu(s,U,X_s)\bigg] \\
    & \quad + \partial_{w z} f(U, X_s, Z_s, W_s) Z_s\frac{1+\pi(s, U, X_s, \mu_s)}{\sigma (s, X_s, [G \mu_s](U))}\bigg[\frac{\int \alpha B^\mu(s,u,x) q(\mathrm{d}\alpha, \mathrm{d}u, \mathrm{d}x) + b(s,X_s,[G\mu_s](U))}{1+\pi(s,U, X_s, \mu_s)} \\
    & \quad \hspace{7.5cm} - B^\mu (s,U,X_s)\bigg].
\end{align*}

The rest of the proof will be divided into several steps.

\vspace{1mm}

\noindent Step 1: Coefficients asymptotics. Then, it follows from \eqref{E:eqn:45}-\eqref{E:eqn:26} together with Lemma \ref{E:lemma:2} that
\begin{equation}
\begin{aligned}
    \delta^{N}_t &= \sup_{1 \leq i \leq N} \left|B^{i,i}_t - \frac{\int \alpha B(t,u,x,\mu^N_t) q^{i,N}_t (\mathrm{d}\alpha, \mathrm{d}u, \mathrm{d}x) + b(t,X_t^i, [G \mu^N_t](u^i))}{1+\pi^i_t} \right| \longrightarrow 0, \\ 
    r^{N}_t &= \sup_{1 \leq i \leq N} \left|\Sigma^{i,i,i}_t - \frac{\sigma(t,X^i_t, [G \mu^N_t](u^i) )}{1+ \pi^i_t}\right| + \sup_{1 \leq i \leq N} \left|\Sigma^{i,i}_t - \frac{\sigma(t,X^i_t, [G \mu^N_t](u^i) )}{1+ \pi^i_t}\right| + \sup_{i \neq j} \left|\Sigma^{i,j}_t\right| \longrightarrow 0.
\end{aligned}
\label{E:eqn:46}
\end{equation}

\vspace{1mm}

\noindent Step 2: Identification of the limit. Let $E:=\mathcal{P}([0,1]^2\times \mathbb{R})$, and denote by $\textbf{M}(E)$ the space of all Borel measures $q(\mathrm{d}t, \mathrm{d}e)$ on $[0, T] \times E$, whose marginal distribution on $[0,T]$ is the Lebesgue measure $\mathrm{d}t$, i.e. $q(\mathrm{d}t, \mathrm{d}e) = q_t(\mathrm{d}e)\mathrm{dt}$ for some family $(q_t)_{t \in [0,T]}$ of Borel probability measures on $E$. We denote by $\Lambda$ the canonical element of $\mathbf{M}(E)$ and set
$$
\Lambda_{t \wedge \cdot} (\mathrm{d}s, \mathrm{d}e):=\Lambda (\mathrm{d}s, \mathrm{d}e) \big|_{[0,t]\times E} + \delta_{e_0}(\mathrm{d}e)\mathrm{d}s \big|_{(t,T] \times E}, \; \text{for some fixed $e_0 \in E$}.
$$
Let $\widetilde{\Omega}:= I \times \mathcal{C}^3 \times \mathbf{M}(E) \times \mathcal{P}(I \times \mathcal{C})$, and denote the corresponding canonical process and canonical filtration by $(\widetilde{U}, \widetilde{X}, \widetilde{Z}, \widetilde{W}, \widetilde{\Lambda}, \widetilde{\mu})$ and $\widetilde{\mathbb{F}}:=(\widetilde{\mathcal{F}}_t)_{t \in [0,T]}$ with $\widetilde{\mathcal{F}}_t := \sigma\{\widetilde{U}, \widetilde{X}_{t \wedge \cdot}, \widetilde{Z}_{t \wedge \cdot}, \widetilde{W}_{t \wedge \cdot}, \widetilde{\Lambda}_{t \wedge \cdot}, \widetilde{\mu}_{t \wedge \cdot}\}$. We also denote by $(\overline{\mu}, \mu)$ the canonical variable on $\mathcal{P}(\widetilde{\Omega}) \times \mathcal{P}(I \times \mathcal{C})$.

We denote for any $N \in \mathbb{N}^*$:
$$
\widehat{\mathrm{P}}^N := \mathcal{L}^\mathbb{P}\big(\overline{\mu}^N, \mu^N\big) \in \mathcal{P}(\mathcal{P}(\widetilde{\Omega}) \times \mathcal{P}(I \times \mathcal{C})),
$$
where $\overline{\mu}^N$ and $\mu^N$ are defined as
$$
\overline{\mu}^N := \frac{1}{N}\sum_{i=1}^N \delta_{(u^i, X^i, Z^i, W^i, \Lambda^i, \mu^N)}, \quad \text{and } \mu^N := \frac{1}{N}\sum_{i=1}^N \delta_{(u^i, X^i)}, \quad \text{where }\Lambda^i (\mathrm{d}q, \mathrm{d}t) := \delta_{q^{i,N}_t} (\mathrm{d}q) \mathrm{d}t.
$$
Under Assumption \ref{E:assump:2}, together with the boundedness of $(b, \sigma)$, we can verify that $(\widehat{\mathrm{P}}^N)_{N \in \mathbb{N}^*}$ is relatively compact in $\mathcal{W}_2$ (Corollary B.2 in \cite{Lacker2015}), and denote by $\widehat{\mathrm{P}}^\infty$ the limit of some sub-sequence. For simplicity, we use the same notation for the sequence and its sub-sequence. We next show that for $\widehat{\mathrm{P}}^\infty$-a.s. $\omega \in \mathcal{P}(\widetilde{\Omega}) \times \mathcal{P}(I \times \mathcal{C})$:
\begin{equation}
    \big(M^{f}_t (\widetilde{U}, \widetilde{X}, \widetilde{Z}, \widetilde{W}, \widetilde{\Lambda}, \widetilde{\mu})\big)_{t \in [0,T]} \text{ is an $(\widetilde{\mathbb{F}}, \overline{\mu}(\omega))$-martingale for all $f \in C_b^2([0,1] \times \mathbb{R}^3)$}.
    \label{E:eqn:33}
\end{equation} 

To prove this, we introduce $M^{f,i}$, the martingale associated to $(u^i, X^i, Z^i, W^i)$ whose dynamics are recalled in \eqref{E:eqn:30}-\eqref{E:eqn:31}:
\begin{align*}
    M^{f, i}_t & := f(u^i, X^i_t, Z^i_t, W^i_t) - \int_0^t\bigg[ \partial_x f(u^i, X_s^i, Z^i_s, W^i_s) B^i_s + \frac{1}{2} \partial_x^2 f(u^i,X^i_s, Z^i_s,W^i_s) \sum_{j=1}^N\left(\Sigma^{i,j}_s\right)^2 \\
    & \qquad + \frac{1}{2} \partial_z^2 f(u^i, X^i_s,Z^i_s,W^i_s) \big(Z^i_s\big)^2 \bigg(\frac{B^{i,i}_s - B^i_s}{\Sigma^{i,i,i}_s}\bigg)^2 + \frac{1}{2} \partial_w^2 f(u^i, X^i_s, Z^i_s, W^i_s) \\
    & \qquad + \partial_{xz} f(u^i, X^i_s, Z^i_s, W^i_s)\Sigma^{i,i}_s Z^i_s  \frac{B^{i,i}_s - B^i_s} {\Sigma^{i,i,i}_s}+ \partial_{xw} f(u^i, X^i_s, Z^i_s, W^i_s) \Sigma^{i,i}_s \\
    & \qquad + \partial_{wz} f(u^i, X^i_s, Z^i_s, W^i_s) Z^i_s  \frac{B^{i,i}_s - B^i_s} {\Sigma^{i,i,i}_s} \bigg]\mathrm{d}s.
\end{align*}
We have that
\begin{align}
    M^{f, i}_t & = f (u^i, X^i_0, Z^i_0, W^i_0) + \bigg[\int_0^t \partial_x f(u^i, X^i_s, Z^i_s, W^i_s) \sum_{j=1}^N \Sigma^{i,j}_s\mathrm{d}W^j_s + \partial_w f(u^i,X^i_s, Z^i_s, W^i_s) \mathrm{d}W^i_s \nonumber \\
    & \hspace{4cm} + \partial_z f(u^i, X^i_s,Z^i_s, W_s^i) Z^i_s \frac{B^{i,i}_s - B^i_s}{\Sigma^{i,i,i}_s}\mathrm{d}W^i_s\bigg].
\label{E:eqn:47}
\end{align}

To handle the singularity appearing in the $M^f_t$, we first notice that for any $h \in C_b ([0,1])$:
\begin{equation*}
    \mathbb{E}^{\widehat{\mathrm{P}}^\infty}\Big[\big|\mathbb{E}^{\overline{\mu}}\big[h(\widetilde{U})\big] -\langle h, \mathrm{d}u\rangle\big|^2\Big]  = \lim_{N \rightarrow \infty} \mathbb{E}^\mathbb{P} \left[\bigg|\frac{1}{N}\sum_{i=1}^N h(u^i) -\langle h, \mathrm{d}u\rangle\bigg|^2\right] \nonumber  =   0,  
\end{equation*}
which further implies that
\begin{equation}
    \overline{\mu}(\omega) \circ (\widetilde{U})^{-1} = \text{Unif}([0,1]), \quad \text{for $\widehat{\mathrm{P}}^\infty$$-$a.s. $\omega \in \mathcal{P}(\widetilde{\Omega}) \times \mathcal{P}(I \times \mathcal{C})$}.
    \label{E:eqn:34}
\end{equation}

By Proposition A.1 of \cite{DjeteTouzi2024}, for $\widehat{\mathrm{P}}^\infty$-a.s. $\omega$ and a.e. $u \in [0,1]$, $\mathcal{L}^{\overline{\mu}(\omega)}(\widetilde{X}_t | \widetilde{U}=u) (\mathrm{d}x)\mathrm{d}t$ has a density w.r.t. the Lebesgue measure on $[0,T] \times \mathbb{R}$. We therefore obtain by Assumption \ref{E:assump:1} (ii) together with Lemma \ref{E:lemma:7} that the discontinuous points of the map $\pi (\cdot)$ are $\overline{\mu}(\omega)$-negligible, for $\widehat{\mathrm{P}}^\infty$-a.s. $\omega$. Let $\Phi:\widetilde{\Omega} \rightarrow \mathbb{R}$ be any bounded continuous function. Then, by the Portmanteau theorem (with $\widetilde{\mu}$ restricted to the set $\mathcal{S}$, defined in \eqref{E:eqn:51}) as well as Lemma \ref{E:lemma:7}, we obtain that
\begin{align}
    & \mathbb{E}^{\widehat{\mathrm{P}}^\infty} \left[\mathbb{E}^{\overline{\mu}} \left[\left(M^{f}_t - M^{f}_s\right) \big(\widetilde{U},\widetilde{X}, \widetilde{Z}, \widetilde{W}, \widetilde{\Lambda}, \widetilde{\mu}\big) \Phi\big(\widetilde{U},\widetilde{X}_{s \wedge \cdot}, \widetilde{Z}_{s \wedge \cdot}, \widetilde{W}_{s \wedge \cdot}, \widetilde{\Lambda}_{s \wedge \cdot}, \widetilde{\mu}_{s \wedge \cdot}\big)\right]^2\right] \nonumber \\
    & =  \lim_{N \rightarrow \infty} \mathbb{E}^{\mathbb{P}}  \left[\bigg|\frac{1}{N}\sum_{i=1}^N\left(M^{f}_t - M^{f}_s\right) \big(u^i,X^i, Z^i, W^i, \Lambda^i, \mu^N\big) \Phi\big(u^i,X^i_{s \wedge \cdot}, Z^i_{s \wedge \cdot}, W^i_{s \wedge \cdot}, \Lambda^i_{s \wedge \cdot}, \mu^N_{s \wedge \cdot}\big)\bigg|^2\right] \nonumber \\
    & = \lim_{N \rightarrow \infty} \mathbb{E}^\mathbb{P} \left[\bigg|\frac{1}{N} \sum_{i=1}^N  \left(M^{f,i}_t - M^{f,i}_s\right) \Phi\big(u^i,X^i_{s \wedge \cdot}, Z^i_{s \wedge \cdot}, W^i_{s \wedge \cdot}, \Lambda^i_{s \wedge \cdot}, \mu^N_{s \wedge \cdot}\big)\bigg|^2\right]  \nonumber \\
    & = 0,
    \label{E:eqn:35}
\end{align}
where we used \eqref{E:eqn:46} for the second equality, and then used \eqref{E:eqn:47}, the independence of $(W^i)_{1 \leq i \leq N}$ as well as Lemma \ref{E:lemma:2} for the last equality. We thus finish verifying \eqref{E:eqn:33} by the arbitrariness of $f$. We can also easily verify that
$$
\overline{\mu}(\omega) \left[\widetilde{\mu} = \mathcal{L}^{\overline{\mu}(\omega)}\big(\widetilde{U}, \widetilde{X}\big) = \mu(\omega)\right] = 1, \quad \text{for $\widehat{\mathrm{P}}^\infty$-a.s. $\omega \in \mathcal{P}(\widetilde{\Omega}) \times \mathcal{P}(I \times \mathcal{C})$}
$$
and $\widetilde{\Lambda}_t (\mathbb{Z}_{\widetilde{\mu}_t}) =1$, $\mathrm{d}t \otimes \mathrm{d} \overline{\mu}(\omega)$-a.s. Consequently, for $\widehat{\mathrm{P}}^\infty$-a.s. $\omega$, possibly on an extension of $(\widetilde{\Omega},\widetilde{\mathcal{F}}, \widetilde{\mathbb{F}}, \overline{\mu}(\omega))$, there exists a martingale measure $\widetilde{M}(\mathrm{d}q, \mathrm{d}t)$ with quadratic variation $\widetilde{\Lambda}$ such that the process $(\widetilde{U}, \widetilde{X}, \widetilde{Z}, \widetilde{W})$ satisfies
$\widetilde{\mu}_t = \mathcal{L}^{\overline{\mu}(\omega)} (\widetilde{U}, \widetilde{X}_t)$ with $\widetilde{U} \sim \text{Unif}([0,1])$, $\widetilde{\Lambda}_t (\mathbb{Z}_{\widetilde{\mu}_t}) =1$, and
\begin{align*}
    & \mathrm{d}\widetilde{X}_t = B(t, \widetilde{U}, \widetilde{X}_t, \widetilde{\mu}_t) \mathrm{d}t + \Sigma(t,\widetilde{U},\widetilde{X}_t, \widetilde{\mu}_t) \mathrm{d}\widetilde{W}_t, \quad \text{with } \widetilde{W}_t = \widetilde{M}(E \times [0,t]), \\
    & \frac{\mathrm{d}\widetilde{Z}_t}{\widetilde{Z}_t} = \int \frac{1+\pi(t, \widetilde{U}, \widetilde{X}_t, \widetilde{\mu}_t)}{\sigma (t, \widetilde{X}_t, [G \widetilde{\mu}_t](\widetilde{U}))} \bigg[\frac{\int \alpha B(t,u,x,\widetilde{\mu}_t) q(\mathrm{d}\alpha, \mathrm{d}u, \mathrm{d}x) + b(t,\widetilde{X}_t,[G\widetilde{\mu}_t](\widetilde{U}))}{1+\pi(t,\widetilde{U}, \widetilde{X}_t, \widetilde{\mu}_t)} \\
    & \hspace{5cm} - B(t,\widetilde{U},\widetilde{X}_t, \widetilde{\mu}_t)\bigg] \widetilde{M}(\mathrm{d}q, \mathrm{d}t).
\end{align*}
In view of the dynamics of $\widetilde{X}$, we can conclude that $\mu(\omega) \in \mathcal{P}_\mathcal{S}(\pi)$, for $\widehat{\mathrm{P}}^\infty$-a.s. $\omega$.

\vspace{1mm}

\noindent Step 3: Characterization of deviating agents. For $\widehat{\mathrm{P}}^\infty$-a.s. $\omega \in \mathcal{P}(\widetilde{\Omega}) \times \mathcal{P}(I \times \mathcal{C})$, let us define the probability
$$
\frac{\mathrm{d}\overline{\mu}^\circ (\omega)}{\mathrm{d}\overline{\mu}(\omega)} := \widetilde{Z}_T,
$$
and the process $(\widetilde{N}_t)_{t \in [0,T]}$ by
\begin{align*}
    & \widetilde{N}_\cdot : = \int_0^\cdot \int \frac{1+\pi(t, \widetilde{U}, \widetilde{X}_t, \widetilde{\mu}_t)}{\sigma (t, \widetilde{X}_t, [G \widetilde{\mu}_t](\widetilde{U}))} \bigg[\frac{\int \alpha B(t,u,x,\widetilde{\mu}_t) q(\mathrm{d}\alpha, \mathrm{d}u, \mathrm{d}x) + b(t,\widetilde{X}_t,[G\widetilde{\mu}_t](\widetilde{U}))}{1+\pi(t,\widetilde{U}, \widetilde{X}_t, \widetilde{\mu}_t)} \\
    & \hspace{5cm} - B(t,\widetilde{U},\widetilde{X}_t, \widetilde{\mu}_t)\bigg] \widetilde{M}(\mathrm{d}q, \mathrm{d}t).
\end{align*}
By Girsanov theorem, $\widetilde{W}_\cdot^\circ := \widetilde{W}_\cdot - \langle \widetilde{W}, \widetilde{N}\rangle_\cdot$ is a $\overline{\mu}^\circ (\omega)$-Brownian motion. It is straightforward that
\begin{align*}
    \mathrm{d}\widetilde{X}_t = \frac{\int \alpha B(t,u,x,\mu_t(\omega)) q(\mathrm{d}\alpha, \mathrm{d}u, \mathrm{d}x) \widetilde{\Lambda}_t(\mathrm{d}q) + b(t, \widetilde{X}_t, [G\mu_t(\omega)](\widetilde{U}))}{1+\pi(t, \widetilde{U}, \widetilde{X}_t, \mu_t(\omega))} \mathrm{d}t + \Sigma(t, \widetilde{U}, \widetilde{X}_t, \mu_t(\omega)) \mathrm{d}\widetilde{W}^\circ_t.
\end{align*}
Thanks to $\widetilde{\Lambda}_t (\mathbb{Z}_{\mu_t(\omega)}) =1$, $\mathrm{d}t \otimes \mathrm{d} \overline{\mu}(\omega)$-a.s. and the Markovian projection techniques, we have $\mathcal{L}^{\overline{\mu}^\circ(\omega)} (\widetilde{U}, \widetilde{X}) \in \mathcal{S} (\pi, \mu(\omega))$. Combining all results, for any bounded continuous functions $\Phi$,
\begin{align*}
    \lim_{N \rightarrow \infty} \frac{1}{N} \sum_{i=1}^N \mathbb{E}^{\mathbb{P}^i_{\Pi, \beta^i}} \left[\Phi(X^{i,i}, \mu^{i,N})\right] & = \lim_{N \rightarrow \infty} \frac{1}{N} \sum_{i=1}^N \mathbb{E}^\mathbb{P} \left[Z^i_T \Phi(X^{i}, \mu^{N})\right]  \\
    & = \lim_{N \rightarrow \infty} \mathbb{E}^{\widehat{\mathrm{P}}^N} \left[\mathbb{E}^{\overline{\mu}} \left[\widetilde{Z}_T \Phi(\widetilde{X},\widetilde{\mu})\right]\right]  \\
    &  = \mathbb{E}^{\widehat{\mathrm{P}}^\infty} \left[\mathbb{E}^{\overline{\mu}} \left[\widetilde{Z}_T \Phi(\widetilde{X}, \widetilde{\mu})\right]\right]  = \mathbb{E}^{\widehat{\mathrm{P}}^\infty} \left[\mathbb{E}^{\overline{\mu}^\circ} \left[\Phi(\widetilde{X},\widetilde{\mu})\right]\right].
\end{align*}

Therefore, we will finish the proof by defining $Q^\infty \in \mathcal{P}(\widetilde{\Omega})$:
$$
Q^\infty := \int_{\mathcal{P}(\widetilde{\Omega}) \times \mathcal{P}(I \times \mathcal{C})} \overline{\mu}^\circ(\omega) \circ \big(\widetilde{U}, \widetilde{X}, \widetilde{Z}, \widetilde{W}^\circ, \widetilde{\Lambda}, \widetilde{\mu}\big)^{-1} \widehat{\mathrm{P}}^\infty (\mathrm{d}\omega),
$$
and the corresponding $\widehat{Q}^\infty \in \mathcal{P}(\widehat{\Omega} \times \mathcal{P}(\widehat{\Omega}))$:
$$
\widehat{Q}^\infty := \int_{\mathcal{P}(\widetilde{\Omega}) \times \mathcal{P}(I \times \mathcal{C})} \overline{\mu}^\circ(\omega) \circ \big(\widetilde{U}, \widetilde{X}, \widetilde{\mu}\big)^{-1} \widehat{\mathrm{P}}^\infty (\mathrm{d}\omega), \quad \text{i.e. $\mathrm{P}^\infty$ in thoerem's statement.}
$$
As for $\widehat{\mathrm{P}}^\infty$-a.s. $\omega$, $\widetilde{\mu} = \mu(\omega)$, $\overline{\mu}(\omega)$-a.s., $\widetilde{W}^\circ$ and $\widetilde{\mu}$ are $Q^\infty$-independent. We conclude the proof by noticing that for $\mathrm{P}^\infty$-a.s. $\omega$,
$$
\mathcal{L}^{\widehat{Q}^\infty}\big(\widehat{U}, \widehat{X} \big| \widehat{\mu}(\omega)\big) = \int_{\mathcal{P}(\widetilde{\Omega}) \times \mathcal{P}(I \times \mathcal{C})} \mathcal{L}^{\overline{\mu}^\circ(\omega^\prime)} \big(\widetilde{U}, \widetilde{X} \big| \widetilde{\mu} = \mu(\omega)\big) \widehat{\mathrm{P}}^\infty(\omega^\prime) \in \mathcal{S} (\pi, \widehat{\mu}(\omega))
$$
since we verified above that $\mathcal{L}^{\overline{\mu}^\circ(\omega^\prime)} \big(\widetilde{U}, \widetilde{X} \big| \widetilde{\mu} = \mu(\omega)\big) \in \mathcal{S} (\pi, \mu(\omega))$, for $\widehat{\mathrm{P}}^\infty$-a.s. $\omega^\prime$.

\subsection{Proof of Theorem \ref{E:thm:2} (i): general kernel} \label{E:subsec:2}
In this framework, we do not have \eqref{E:eqn:44}. Hence, \eqref{E:eqn:45}-\eqref{E:eqn:26} are replaced by
\begin{align}
    & B^i_t = \frac{c(t, \mu^N_t) + b (t,X^i_t, [G \mu^N_t](u^i))}{1+\pi^i_t} + \frac{\mathcal{H}^{1,i}_t}{1+\pi^i_t} = B(t,u^i, X^i_t, \mu^N_t) + \frac{\mathcal{H}^{1,i}_t}{1+\pi^i_t}, \label{E:eqn:57} \\
    & \bigg(1 + \frac{N-1}{N} \pi^i_t + \frac{1}{N} \beta^{i,i}_t\bigg) B^{i,i}_t = \int \alpha B(t,u,x,\mu^N_t) q^{i,N}_t (\mathrm{d}\alpha, \mathrm{d}u, \mathrm{d}x) + b(t,X_t^i, [G \mu^N_t](u^i)) + \mathcal{H}^{2,i}_t, \\
    & \Sigma^{i,j}_t = \frac{\frac{1}{N}A^j_t(\mathbf{u},\mathbf{X}_t)\sigma(t,X^j_t, [G^N \mu^N_t](u^j))}{1+\pi^i_t} \longrightarrow 0, \quad \text{for } j \neq i, \\
    & \Sigma^{i,i}_t = \Sigma (t, u^i, X^i_t, \mu^N_t) + \mathcal{H}^{3,i}_t, \\
    & \bigg(1 + \frac{N-1}{N}\pi^i_t + \frac{1}{N} \beta^{i,i}_t\bigg) \Sigma^{i,i,i}_t = \sigma(t,X^i_t, [G \mu^N_t](u^i)) + \mathcal{H}^{4,i}_t, \label{E:eqn:58}
\end{align}
where $\mathcal{H}^{k,i}_t$, $k=1, \ldots,4$ are defined as
\begin{align*}
    & \mathcal{H}^{1,i}_t := \frac{\frac{1}{N}\sum_{j=1}^N \frac{\pi^j_t}{1+\pi^j_t} \left(b(t, X^j_t, [G^N \mu^N_t](u^j)) - b(t, X^j_t, [G \mu^N_t](u^j))\right)}{1-\frac{1}{N}\sum_{k=1}^N\frac{\pi^k_t}{1+\pi^k_t}}\\
    & \hspace{1.2cm} + b (t,X^i_t, [G^N \mu^N_t](u^i)) - b (t,X^i_t, [G \mu^N_t](u^i)), \\
    & \mathcal{H}^{2,i}_t := b(t,X_t^i, [G^N \mu^N_t](u^i)) - b(t,X_t^i, [G \mu^N_t](u^i)) + \frac{1}{N}\sum_{j=1}^N \beta^{i,j}_t (B^{i,j}_t - B^j_t) + \frac{1}{N}\sum_{j=1}^N \beta^{i,j}_t\frac{\mathcal{H}^{1,j}_t}{1+\pi^j_t}, \\
    & \mathcal{H}^{3,i}_t := \frac{\left(\frac{1}{N}A^i_t(\mathbf{u},\mathbf{X}_t)+1\right) \sigma(t,X^i_t, [G^N \mu^N_t](u^i)) - \sigma(t,X^i_t, [G \mu^N_t](u^i))}{1+\pi^i_t}, \\
    & \mathcal{H}^{4,i}_t :=\frac{1}{N}\sum_{j=1}^N \beta^{i,j}_t \Sigma^{i,j,i}_t+\sigma(t,X^i_t, [G^N \mu^N_t](u^i)) - \sigma(t,X^i_t, [G \mu^N_t](u^i)).
\end{align*}

Coefficients asymptotics are accordingly changed to the following. We start by introducing
\begin{align*}
    \delta^{i,N}_t &= \big|B^i_t - B(t,u^i, X^i_t, \mu^N_t)\big| +\big|\Sigma^{i,i}_t - \Sigma(t,u^i,X^i_t, \mu^N_t)\big| + \sum_{i \neq j} \big|\Sigma^{i,j}_t\big|, \\
    r^{i,N}_t &= \bigg|\frac{B^{i,i}_t - B^i_t}{\Sigma^{i,i,i}_t} -\frac{\int \alpha B(t,\overline{x},\mu^N_t) q^{i,N}_t (\mathrm{d}\alpha, \mathrm{d}\overline{x}) + b(t,X_t^i, [G \mu^N_t](u^i)) -(1+\pi^i_t)B(t,u^i,X^i_t,\mu^N_t)}{\sigma(t, X^i_t, [G \mu^N_t](u^i))} \bigg|.
\end{align*}
Using \eqref{E:eqn:57}-\eqref{E:eqn:58}, together with the results proved in Lemmas \ref{E:lemma:2} and \ref{E:lemma:4}, we obtain
\begin{align}
    \max_{1 \leq i \leq N}r^{i,N}_t +  \max_{1 \leq i \leq N}\delta^{i,N}_t &\leq C \max_{1 \leq i \leq N}\bigg(\big|\mathcal{H}^{1,i}_t\big| +  \big|\mathcal{H}^{2,i}_t\big| +\big|\mathcal{H}^{3,i}_t\big| +\big|\mathcal{H}^{4,i}_t \big|+ \frac{1}{N}\bigg) \nonumber \\
    & \leq C \bigg\{\bigg[\bigg(N^2\big\|G^N -G \big\|_\square +N^2 \sup_{1 \leq i,j \leq N} \mathcal{R}^N_{ij}\bigg) \bigg(1+\frac{1}{N}\sum_{k=1}^N \big|X^k_t\big|^2\bigg)\bigg]^{1/2} + \frac{1}{N}\bigg\}. \label{E:eqn:32}
\end{align} 
Notice that 
$$
N^2\big\|G^N -G \big\|_\square +N^2 \sup_{1 \leq i,j \leq N} \mathcal{R}^N_{ij} \longrightarrow 0
$$
since $N^2 \|G^N -G \|_\square \rightarrow 0$ and $G$ is continuous on $[0,1]^2$. The upper bound gave in \eqref{E:eqn:32} will only be used to ensure the following equality,  
\begin{align}
    & \lim_{N \rightarrow \infty} \mathbb{E}^{\mathbb{P}}  \left[\bigg|\frac{1}{N}\sum_{i=1}^N\left(M^{f}_t - M^{f}_s\right) \big(u^i,X^i, Z^i, W^i, \Lambda^i, \mu^N\big) \Phi\big(u^i,X^i_{s \wedge \cdot}, Z^i_{s \wedge \cdot}, W^i_{s \wedge \cdot}, \Lambda^i_{s \wedge \cdot}, \mu^N_{s \wedge \cdot}\big)\bigg|^2\right] \nonumber \\
    & = \lim_{N \rightarrow \infty} \mathbb{E}^\mathbb{P} \left[\bigg|\frac{1}{N} \sum_{i=1}^N  \left(M^{f,i}_t - M^{f,i}_s\right) \Phi\big(u^i,X^i_{s \wedge \cdot}, Z^i_{s \wedge \cdot}, W^i_{s \wedge \cdot}, \Lambda^i_{s \wedge \cdot}, \mu^N_{s \wedge \cdot}\big)\bigg|^2\right]  \nonumber 
\end{align}
The rest of the proof is the same as Section \ref{E:subsec:1}.

\bibliographystyle{plain}
\bibliography{Mutual_holding}

\end{document}